\documentclass[11pt]{amsart}

\usepackage[margin=1in]{geometry}
\usepackage{amsthm,amssymb}
\usepackage{amsmath}
\usepackage{tikz}
\usepackage{enumerate}
\usepackage[colorlinks=true,linkcolor=blue,anchorcolor=blue,citecolor=blue,filecolor=blue,menucolor=blue,runcolor=blue,urlcolor=Blue]{hyperref}
\usepackage{tikz-cd}
\usepackage{mathtools}
\mathtoolsset{showonlyrefs} 
\usepackage{mathrsfs}
\usepackage[noadjust]{cite}
\usepackage{placeins}
\usepackage{enumerate}

\newtheorem{proposition}{{\textbf{Proposition}}}
\newtheorem{theorem}{{\textbf{Theorem}}}
\newtheorem{lemma}{{\textbf{Lemma}}}
\newtheorem{corollary}{{\textbf{Corollary}}}

\theoremstyle{definition}
\newtheorem{definition}{{\textbf{Definition}}}

\theoremstyle{remark}
\newtheorem{example}{Example}
\newtheorem{warning}{Warning}
\newtheorem{rem}{{Remark}}

\DeclarePairedDelimiter\floor{\lfloor}{\rfloor}

\newcommand{\ZZ}{\mathbb Z}
\newcommand{\QQ}{\mathbb Q}
\newcommand{\RR}{\mathbb R}
\newcommand{\CC}{\mathbb C}
\renewcommand{\AA}{\mathbb A} 
\newcommand{\PP}{\mathbb P}

\newcommand{\GG}{\mathbb G}
\renewcommand{\a}{\mathfrak a}

\renewcommand{\pmod}[1]{\ (\mathrm{mod}\ #1)}

\newcommand{\bparen}[1]{\left(#1\right)}

\begin{document}

\title[Monic abelian trace-one cubics]{On monic abelian trace-one cubic polynomials}

\author{Shubhrajit Bhattacharya}
\address{
\parbox{0.5\linewidth}{
    Department of Mathematics\\
    University of British Columbia\\
    Vancouver, BC Canada \\[.2em]
    }
}
\email{shubhrajit@math.ubc.ca}
\author{Andrew O'Desky}
\address{
\parbox{0.5\linewidth}{
    Department of Mathematics\\
    Princeton University\\
    Princeton, NJ USA\\[.2em]
    }
}
\email{andy.odesky@gmail.com}

\date{October 15, 2023}

\begin{abstract}
We compute the asymptotic number 
of monic trace-one integral polynomials with 
Galois group $C_3$ and bounded height. 
For such polynomials 
we compute a height function 
    coming from toric geometry 
    and introduce a parametrization using 
    the quadratic cyclotomic field $\QQ(\sqrt{-3})$. 
We also give a formula for 
    the number of polynomials of the form 
    $t^3 -t^2 + at + b \in \ZZ[t]$  
    with Galois group $C_3$ for a fixed integer $a$. 
\end{abstract}

\maketitle

\section{Introduction} 

Let $F$ denote the set of polynomials 
    of the form 
    $t^3 -t^2 + at + b \in \ZZ[t]$ 
    which have Galois group $C_3$, 
    the cyclic group of order three. 
The primary aim of this paper is to prove 
    the following asymptotic formula. 

\begin{theorem}\label{thm:asymptoticPolynomialCount} 
Let $\varepsilon >0$. 
The number of polynomials 
    $t^3 -t^2 + at + b \in F$ 
    with $\max(|a|^{1/2},|b|^{1/3})\leq H$ 
    is equal to 
    $$
    CH^2 \log H + \left(
    C\log \sqrt{3}+D
    -\frac{\pi}{3\sqrt{3}}\right)H^2
    + O_\varepsilon(H^{1+\varepsilon})
    $$
    as $H \to \infty$, where 
\begin{equation} 
    C=
        \frac{4\pi^2}{81}
        \prod_{q\equiv2\pmod{3}}\left(1-\frac{1}{q^{2}}\right)
        \prod_{p\equiv1\pmod{3}}\left(1-\frac{3}{p^{2}}+\frac{2}{p^{3}}\right)
\end{equation} 
and 
\begin{equation} 
    \frac{D}{C} = 
    2 \gamma + \log(2\pi) 
    - 3 \log\left(\frac{\Gamma(1/3)}{\Gamma(2/3)}\right)
    +
    \frac{9}{8}\log 3
    +
    \frac{9}{4}
    \sum_{q\equiv2\pmod{3}}
        \frac{\log q}{q^2-1}
    +
    \frac{27}{4}
    \sum_{p\equiv1\pmod{3}}
        \frac{(p+1)\log p}{p^3-3p+2} .
\end{equation} 
\end{theorem} 

This may be qualitatively compared with 
    \cite[Theorem~1.1]{xiao_2022} 
    which asserts that the number $N(H)$ 
    of monic integral cubic polynomials 
    $t^3 +at^2 + bt + c$ 
    with Galois group $C_3$ 
    and $\max(|a|,|b|,|c|) \leq H$ 
    satisfies 
    $2H \leq N(H) \ll H (\log H)^2$, 
    however their height function is inequivalent 
    to the height in 
    Theorem~\ref{thm:asymptoticPolynomialCount} 
    and there is no trace-one condition. 

We also prove a formula of sorts 
    for the number of $f \in F$ with 
    specified nonconstant coefficients. 

\begin{theorem}\label{thm:exactFormula}
For any $H \geq 1$ 
    let $E_H \subset \RR^2$ be the ellipse defined by 
\begin{equation} 
    E_{H} : 
    x^2+y^2+xy-x-y = \tfrac{1}{3}(H^2-1).
\end{equation} 
If $t^3 -t^2 + at + b \in F$ then $a \leq 0$. 
Fix $a \in \ZZ_{\leq 0}$. 
The number of polynomials of the form 
    $t^3 -t^2 + at + b \in F$ 
    for any $b \in \ZZ$ is equal to 
\begin{equation} 
    \frac12\sum_{d|(1-3a)}
    3^{\omega(P_1(d))}
    (-1)^{\Omega(P_2(d))}
    -\frac16 \#E_{\sqrt{1-3a}}(\ZZ)
\end{equation} 
    where 
    $P_j(d)$ denotes the largest divisor of $d$ 
    only divisible by primes $\equiv j \pmod 3$, 
    and $\omega(n)$ (resp. $\Omega(n)$) 
    denotes the number of prime factors of a positive integer $n$ 
    counted without (resp. with) multiplicity. 
\end{theorem}


\subsection*{An integral Diophantine problem} 

To prove these theorems we relate 
    the polynomial counting problem to an 
    integral Diophantine problem 
    on a certain singular toric surface $S$ 
    and then solve the Diophantine problem. 
Let $\AA^3 = \mathrm{Spec}\, \QQ[X,Y,Z]$ 
    and $\mathbb P_2 = \PP(\AA^3)= \mathrm{Proj}\, \QQ[X,Y,Z]$ 
    be equipped with the regular action of $C_3$. 
Consider the quotient surface 
\begin{equation} 
    S = \PP_2/C_3. 
\end{equation} 
Let $T \subset S$ denote the image of the unit group 
    in the group algebra $\AA^3$ of $C_3$ 
    under $\AA^3-\{0\} \to \PP_2 \to S$. 
One can show that $T$ is a rank-two torus and $S$ is 
    a toric compactification of $T$. 
The set of rational points $S(\QQ)$ is thus equipped with 
    a family of \emph{toric height functions} 
    $H(-,s)$ 
    constructed in \cite{MR1369408}, 
    where $s$ is a parameter in the complexified Picard group 
    $\mathrm{Pic}(S) \otimes \CC$. 
The surface $S$ has Picard rank one  
    \cite[Corollary~3.6]{odesky2023moduli}, 
    so we may regard $s$ as 
    a complex number where $s = 3$ 
    corresponds to the ample generator. 
Let $D_0$ be the divisor 
    $\{\varepsilon \coloneqq X+Y+Z = 0\} \subset S$. 
A rational point $P$ of $S - D_0$ is 
\textit{$D_0$-integral} 
if every regular function in 
$\mathcal O(S_\ZZ - D_0) 
= \mathbb \ZZ[X/\varepsilon,Y/\varepsilon]^{C_3}$ 
is $\ZZ$-valued on $P$. 

Our third result is an explicit formula 
    for the height zeta function for $D_0$-integral 
    rational points on the torus $T \subset S$. 

\begin{theorem}\label{thm:heightZetaFunctionFormula}
\begin{equation} 
    \sum_{\substack{P \in T(\mathbb Q),\\\text{$D_0$-\emph{integral}}}}
        H(P,s)^{-1}
        =
    \left(1-\frac{1}{3^z}\right)^2
        \zeta_{\QQ(\sqrt{-3})}(z)^2
        \prod_{q\equiv2\pmod{3}}\left(1-\frac{1}{q^{2z}}\right)
        \prod_{p\equiv1\pmod{3}}\left(1-\frac{3}{p^{2z}}+\frac{2}{p^{3z}}\right) 
\end{equation} 
    where $z = \tfrac {s}2 $ and 
    $\zeta_{\QQ(\sqrt{-3})}$ is the Dedekind zeta function 
    of $\QQ(\sqrt{-3})$. 
This height zeta function 
    can be meromorphically continued to the half-plane $\mathrm{Re}(s)>1$ 
    and its only pole in this region is at $s = 2$ with order $2$. 
If $n \in \ZZ_{\geq 1}$ is not divisible by $3$, 
    then the number of $D_0$-integral rational points on $T$ 
    with toric height $\sqrt{n}$ 
    is equal to 
\begin{equation} 
    \sum_{d|n}
    3^{\omega(P_1(d))}
    (-1)^{\Omega(P_2(d))}
\end{equation} 
\end{theorem}


\subsection*{Relation between the problems} 

In \cite{odesky2023moduli} 
    it was shown that the torus $T$ 
    is the moduli space for $C_3$-algebras 
    with a given trace-one normal element. 
In particular, 
\begin{equation} 
    T(\QQ) \cong 
    \{(K/\QQ \text{ $C_3$-algebra},\,\,
    x \text{ trace-one normal})\}
\end{equation} 
where a \emph{$C_3$-algebra} $K/\QQ$ is 
a $\QQ$-algebra equipped with an action of $C_3$ 
for which there is a $C_3$-linear $\QQ$-algebra isomorphism 
from $K$ to either a cubic abelian number field or 
the split algebra $\QQ^3$, 
and an element $x \in K$ is \emph{normal} if 
its Galois conjugates 
are linearly independent over $\QQ$. 
Using this bijection we consider the function 
\begin{equation} 
    T(\QQ) 
    \,\longrightarrow
    \{t^3 -t^2 + at + b \in \QQ[t]\}
\end{equation} 
taking a rational point $(K/\QQ,x)$ 
to the characteristic polynomial of $x$. 
We prove that the image of this function is 
    the subset of polynomials 
    which either have Galois group $C_3$ 
    or split into three linear factors over $\QQ$ 
    with at most two being the same, 
and if $f$ is such a polynomial, then 
    the number of rational points of $T$ 
    with characteristic polynomial $f$ is given by 
\begin{equation}\label{eqn:polynomialWeights} 
    w_f=
    \begin{cases}
        1 &\text{if $f$ has a double root,}\\
        2 &\text{otherwise.}\\
    \end{cases}
\end{equation} 
Moreover we show that 
    a rational point $P$ of $T$ is $D_0$-integral 
    if and only if 
    the associated characteristic polynomial 
    $t^3 -t^2 + at + b$ is integral, 
and we also prove that 
\begin{equation} 
    H(P,1) = \sqrt{1-3a}
\end{equation} 
for $D_0$-integral points. 
This toric height is equivalent to the height used in 
    Theorem~\ref{thm:asymptoticPolynomialCount}. 


\subsection*{Further remarks} 

The restriction to trace-one normal elements 
    was made out of convenience 
    in \cite{odesky2023moduli} 
    and should not be essential for the method. 
In place of $S$, there is a three-fold 
    with a similar construction 
    and an open subset which parametrizes 
    all normal elements of $C_3$-algebras. 
In forthcoming work \cite{odeskyNP} 
    the method presented here will be extended 
    to count monic integral polynomials 
    with bounded height and any given abelian Galois group. 


\subsection*{Acknowledgements} 

A.O. is very grateful to 
Timothy Browning, 
Vesselin Dimitrov, 
Jef Laga, 
Peter Sarnak, 
Sameera Vemulapalli, 
Victor Wang, 
and Shou-Wu Zhang 
for helpful discussions 
and comments on an earlier draft. 
A.O. would also like to thank Alexandra Pevzner 
    for pointing out the reference \cite{campbell2023permutation}. 
A.O. was supported by NSF grant DMS-2103361. 



\section{The orbit parametrization}\label{sec:orbitParam} 

In this section we recall some facts from \cite{odesky2023moduli} 
    and describe the orbit parametrization. 
Let $\sigma$ be a generator of $C_3$. 
Let $\Delta =3 X Y Z-X^3-Y^3-Z^3$, the determinant of 
    multiplication by an element 
    $Xe+Y\sigma+Z\sigma^2$ of the group algebra. 
We set 
$$\mathcal G = \mathbb P_2[\Delta^{-1}]
\quad\text{and}\quad
T = \mathcal G/C_3.$$ 
Then $\mathcal G$ is an algebraic torus over $\QQ$ 
which may be identified 
with the units of the group algebra of $C_3$ with augmentation one, 
i.e.~$$\mathcal G = \{(x,y,z) \in \AA^3 : 
\Delta(x,y,z) \in \GG_m \text{ and } x+y+z = 1\}.$$ 
Since $C_3$ is abelian, the homogeneous space $T=\mathcal G/C_3$ 
    is itself an algebraic torus over $\QQ$. 
The action of $\mathcal G$ on the regular representation 
induces an action of $T$ on $S$ extending 
the regular action of $T$ on itself. 
Let $\AA^2 
= \mathrm{Spec}\,\mathbb Q[X/\varepsilon,Y/\varepsilon]$ 
denote the open affine plane in $\PP_2$ where 
the augmentation map $\varepsilon=X+Y+Z$ is nonvanishing. 
A rational (or adelic) point $P$ of $\AA^2/C_3$ is 
\textit{$D_0$-integral} 
if every regular function in 
$\mathcal O(\AA_\ZZ^2/C_3) 
= \mathbb \ZZ[X/\varepsilon,Y/\varepsilon]^{C_3}$ 
is $\ZZ$-valued (resp.~$\widehat{\ZZ} \times \RR$-valued) on $P$. 

\subsection{\texorpdfstring{$T$}{T} as a moduli space} 

Let $K/\mathbb Q$ be a separable $\mathbb Q$-algebra 
    equipped with the action of 
    a finite group $G$ of $\mathbb Q$-algebra automorphisms of $K$. 
We say that $K/\mathbb Q$ regarded with its $G$-action is a 
    \textit{(Galois) $G$-algebra} if 
    the subset of $K$ fixed by $G$ is equal to $\mathbb Q$. 
Geometrically, a $G$-algebra is the ring of functions 
    on a principal $G$-bundle, 
    equipped with its natural $G$-action. 

\begin{warning} 
Since we regard the $G$-action as part of the data of a $G$-algebra, 
    a $G$-algebra 
    is not generally determined by the isolated data of 
    the underlying $\QQ$-algebra $K$ and 
    the abstract finite group $G$. 
The $G$-action on a $G$-algebra 
    may be twisted by any outer automorphism of $G$, 
    and the twisted $G$-algebra will not generally 
    be isomorphic to the original $G$-algebra. 
\end{warning} 

Two pairs $(K/\mathbb Q,x)$, 
    $(K'/\mathbb Q,x')$ are regarded as {equivalent} if 
    there is a $G$-equivariant $\mathbb Q$-algebra isomorphism 
    $K \to K'$ sending $x$ to $x'$. 
We make use of the following modular interpretation for $T$. 

\begin{theorem}[{\cite[\S2]{odesky2023moduli}}]\label{thm:modularInterpretation}
    The homogeneous variety $T$ is the moduli space for 
    $C_3$-algebras with a given trace-one normal element. 
In particular, 
    there is a bijection 
    between rational points of $T$ 
    and equivalence classes of 
    $C_3$-algebras $K/\mathbb Q$ equipped with 
    a trace one normal element $x \in K$. 
\end{theorem}


\begin{example}\label{example:twoRationalPointsSamePoly} 
Let $K$ be a cubic abelian number field. 
Then $K$, equipped with its canonical Galois action, 
    is a $C_3$-algebra. 
The twist $K'$ of the $C_3$-algebra $K$ by 
    the outer automorphism $g \mapsto g^{-1}$ of $C_3$ 
    (with twisted action $g \ast x = g^{-1}x$) 
    is not isomorphic to $K$ as a $C_3$-algebra.\footnote{ 
In terms of Galois cohomology, 
    the non-cohomologous $1$-cocycles in $H^1(\QQ,C_3)$ 
    corresponding to the $C_3$-algebras $K$ and $K'$ 
    have the same image under the canonical map 
    $H^1(\QQ,C_3) \to H^1(\QQ,S_3)$ because the outer automorphism 
    of $C_3$ is realized by $S_3$-conjugation.} 
\end{example} 

\begin{example}\label{example:splitAlgebra} 
Let $K_{\text{spl}} = \QQ^3$, the split cubic algebra. 
Then $C_3\subset S_3
    = \mathrm{Aut}_{\text{$\QQ$-alg}}(K_{\text{spl}})$ 
and $K_{\text{spl}}$, equipped with its canonical $C_3$-action, 
    is a $C_3$-algebra. 
Any transposition gives an isomorphism of $C_3$-algebras 
    from $K_{\text{spl}}$ to its twist $K_{\text{spl}}'$. 
\end{example} 

\begin{example}\label{example:splitAlgebraNormalElement} 
An element $x$ of the split $C_3$-algebra $K_{\text{spl}}$ 
    is normal 
    if and only if 
    $x$ either has distinct coordinates or 
    exactly two identical coordinates. 
The pairs $(K_{\text{spl}},x)$ and $(K_{\text{spl}}',x)$ 
    are equivalent 
    if and only if 
    $x$ has exactly two identical coordinates 
    (swapping the identical coordinates 
    gives the required isomorphism); 
in particular, if $x$ has distinct coordinates 
    then $(K_{\text{spl}},x)$ and $(K_{\text{spl}}',x)$ 
    determine different rational points of $\mathcal G/C_3$, 
    even though $K_{\text{spl}}$ and $K_{\text{spl}}'$ 
    are isomorphic as $C_3$-algebras. 
\end{example} 


\subsection{\texorpdfstring{$T$}{T} as a torus}\label{sec:ToricDescription} 

Here we describe some of the toric data associated 
with $T$ which will be needed later. 
For more details see e.g. \cite[p.~202]{MR1654185}. 
Let $E = \mathbb Q(\zeta)$ 
where $\zeta$ is a primitive cube root of unity, 
and let $\gamma$ denote the generator of 
the Galois group $\Gamma$ of $E$ over $\QQ$. 
Let $Pl_E$ denote the set of places of $E$. 
The group of units $U$ in the group algebra 
    is a three-dimensional algebraic torus defined over $\mathbb Q$ 
    which canonically factors as $U = \GG_m \times \mathcal G$. 
The characters and cocharacters of $T$ may be described as follows. 
The larger torus $U$ is diagonalized over $E$ 
by the three elementary idempotents in the group algebra: 
\begin{align*}
    v_0' =\tfrac{1}{3}(1+\sigma+\sigma^2),\quad
    v_1' =\tfrac{1}{3}(1+\zeta^2 \sigma+\zeta \sigma^2), \quad
    v_2' =\tfrac{1}{3}(1+\zeta \sigma+\zeta^2\sigma^2).
\end{align*}
Each idempotent is associated with 
a character 
    $\chi_i \colon U(E) \to E^\times$ 
    for $i = 0,1,2$ 
determined by $u v_i' = \chi_i(u)v_i'$, 
corresponding to the action of $U$ on 
the $i$th irreducible representation of $C_3$. 
The character $\chi_0$ is trivial on $\mathcal G$, so 
the lattice of characters of $\mathcal G_E$ is generated 
by $\chi_1$ and $\chi_2$. 
We denote this lattice by $M_E'$ 
and let $N_E'$ denote the dual lattice to $M_E'$. 
To describe the fans 
it is more symmetric to work with the isomorphic image 
of $N_E'$ in the quotient of $\CC C_3$ by the line spanned by 
$v_0'+v_1'+v_2'$,  
and we write $v_i$ for the image of $v_i'$ ($i = 0,1,2$). 
The Galois group $\Gamma$ of $E$ acts on $M_E'$ 
    by swapping $\chi_1$ and $\chi_2$, 
    and on $N_E'$ via the dual action. 

To pass from $\mathcal G$ to $T$, consider the element 
\begin{equation} 
    \omega=\tfrac13(2v_1+v_2) \in N_{E,\QQ}'
\end{equation} 
and set 
\begin{equation} 
    N_E = N_E' + \omega 
    \quad\text{and}\quad
    M_E = N_E^\vee 
    = \{m \in M_{E,\QQ}' : 
    m(n) \in \ZZ \text{ for all $n \in N_E$}\}.
\end{equation} 
The character lattice (resp.~cocharacter lattice) 
of $T_E$ is $M_E$ (resp.~$N_E$). 
The cocharacters $\omega$ and $\gamma \omega$ span $N_E$ 
    so the dual basis $(a',b') =(\omega,\gamma \omega)^\vee$ 
    spans $M_E$. 

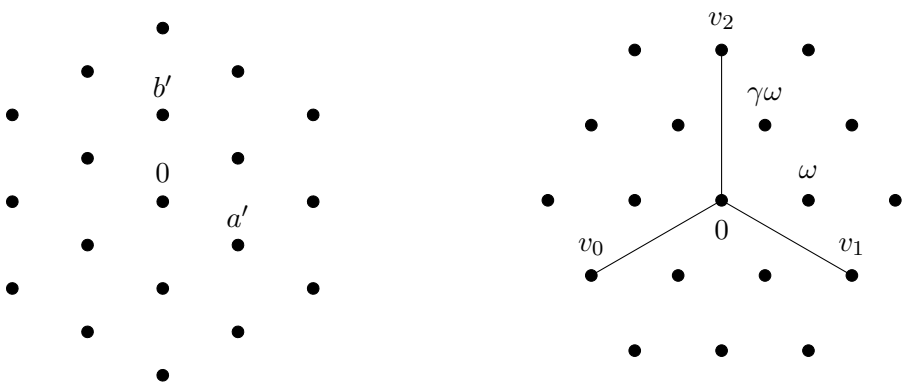
\begin{figure}[h!] 
\begin{minipage}{.45\textwidth} 
    \centering
\vspace{1.5em}
\begin{tikzpicture}[scale=2,rotate=30] 

\filldraw [black] (0,0) circle (.1em);
\filldraw [black] ({1/sqrt(3)},0) circle (.1em);
\filldraw [black] ({2/sqrt(3)},0) circle (.1em);
\filldraw [black] ({-1/sqrt(3)},0) circle (.1em);
\filldraw [black] ({-2/sqrt(3)},0) circle (.1em);
\filldraw [black] ({1/sqrt(3)+1/(2*sqrt(3))},1/2) circle (.1em);

\filldraw [black] ({1/(2*sqrt(3))},1/2) circle (.1em);
\filldraw [black] ({2/(2*sqrt(3))},1) circle (.1em);
\filldraw [black] ({-1/(2*sqrt(3))},-1/2) circle (.1em);
\filldraw [black] ({-2/(2*sqrt(3))},-1) circle (.1em);
\filldraw [black] ({-1/(2*sqrt(3))},1/2) circle (.1em);
\filldraw [black] ({-2/(2*sqrt(3))},1) circle (.1em);
\filldraw [black] ({1/(2*sqrt(3))},-1/2) circle (.1em);
\filldraw [black] ({2/(2*sqrt(3))},-1) circle (.1em);

\filldraw [black] ({sqrt(3)/2},-1/2) circle (.1em); 
\filldraw [black] ({-sqrt(3)/2},-1/2) circle (.1em);
\filldraw [black] (0,1) circle (.1em);

\filldraw [black] ({-sqrt(3)/2},1/2) circle (.1em);
\filldraw [black] (0,-1) circle (.1em);
 
\node[yshift=1em] at ({1/(2*sqrt(3))},1/2) {$b'$};
\node[yshift=1em] at ({1/(2*sqrt(3))},-1/2) {$a'$};
\node[yshift=1em] at (0,0) {$0$};

\end{tikzpicture} 
\end{minipage}%
\begin{minipage}{.45\textwidth} 
    \centering
\begin{tikzpicture}[scale=2] 

\filldraw [black] (0,0) circle (.1em);
\filldraw [black] ({1/sqrt(3)},0) circle (.1em);
\filldraw [black] ({2/sqrt(3)},0) circle (.1em);
\filldraw [black] ({-1/sqrt(3)},0) circle (.1em);
\filldraw [black] ({-2/sqrt(3)},0) circle (.1em);
\filldraw [black] ({1/sqrt(3)+1/(2*sqrt(3))},1/2) circle (.1em);

\filldraw [black] ({1/(2*sqrt(3))},1/2) circle (.1em);
\filldraw [black] ({2/(2*sqrt(3))},1) circle (.1em);
\filldraw [black] ({-1/(2*sqrt(3))},-1/2) circle (.1em);
\filldraw [black] ({-2/(2*sqrt(3))},-1) circle (.1em);
\filldraw [black] ({-1/(2*sqrt(3))},1/2) circle (.1em);
\filldraw [black] ({-2/(2*sqrt(3))},1) circle (.1em);
\filldraw [black] ({1/(2*sqrt(3))},-1/2) circle (.1em);
\filldraw [black] ({2/(2*sqrt(3))},-1) circle (.1em);

\filldraw [black] ({sqrt(3)/2},-1/2) circle (.1em); 
\filldraw [black] ({-sqrt(3)/2},-1/2) circle (.1em);
\filldraw [black] (0,1) circle (.1em);

\filldraw [black] ({-sqrt(3)/2},1/2) circle (.1em);
\filldraw [black] (0,-1) circle (.1em);
 
\draw (0,0) -- ({sqrt(3)/2},-1/2);
\draw (0,0) -- (0,1);
\draw (0,0) -- ({-sqrt(3)/2},-1/2);

\node[yshift=1em] at ({1/sqrt(3)},0) {$\omega$};
\node[yshift=1em] at ({1/(2*sqrt(3))},1/2) {$\gamma \omega$};
\node[yshift=1em] at ({sqrt(3)/2},-1/2) {$v_1$};
\node[yshift=1em] at (0,1) {$v_2$};
\node[yshift=1em] at (-{sqrt(3)/2},-1/2) {$v_0$};
\node[yshift=-1em] at (0,0) {$0$};


\end{tikzpicture}  
\end{minipage}
\vspace{.5cm}
\caption{The dual lattice $M_E = \ZZ \langle a',b' \rangle$ (left) and the fan $\Sigma$ of $S$ in $N_{E,\RR}$ (right).}
\end{figure}

The fan $\Sigma$ of $S$ is the same as the fan for $\PP_2$ 
and has three generators $\Sigma(1) = \{v_0,v_1,v_2\}$. 

We also make use of the following formulas for 
    the characters of $\mathcal G$. 
Let $(v_1^\vee,v_2^\vee) \in M_E'$ be 
    the dual basis to $(v_1,v_2) \in N_E'$. 
The characters of $\mathcal G$ associated to 
$v_1^\vee$ and $v_2^\vee$ are given on $E$-points of $U$ by 
\begin{equation} 
    \chi^{v_1^\vee}(uv_0' + vv_1'+wv_2') = \frac vu
    \quad\text{and}\quad
    \chi^{v_2^\vee}(uv_0' + vv_1'+wv_2') = \frac wu.
\end{equation} 

This explicit description of the character lattices 
    leads to an (unexpected) isomorphism between 
    $\mathcal G$ and its quotient $T = \mathcal G/C_3$. 
On character lattices 
    it is given by the $\Gamma$-equivariant isomorphism 
\begin{equation}\label{eqn:isoCocharLattices} 
    N_E=\ZZ\langle\omega,\gamma\omega \rangle 
    \to N_E'=\ZZ\langle v_1,v_2\rangle
\end{equation} 
    taking $\omega$ to $v_1$ and $\gamma \omega$ to $v_2$. 
This implies that the multiplicative group of 
    the cyclotomic field $\QQ(\sqrt{-3})$ 
    naturally parametrizes cubic trace-one polynomials. 

\begin{proposition}\label{prop:cyclotomicParametrization}
The tori $T$ and $\mathcal G = R^E_\QQ \mathcal \GG_m$ 
    are isomorphic as algebraic groups over $\QQ$. 
Every rational point $(K/\QQ,x)$ of $T$ thereby 
    determines an element of $\QQ(\sqrt{-3})^\times$ 
    which is canonically determined up to the action of 
    $\mathrm{Aut}(\mathcal G)$. 
The toric height $H(f)\coloneqq \sqrt{1-3a}$ 
    on $T(\QQ)$ is identified with 
    the square-root of the norm on $\QQ(\sqrt{-3})^\times$. 
Let $\zeta$ be a primitive cube root of unity. 
If $u+v\zeta \in \QQ(\sqrt{-3})^\times$ 
has norm $N$ and trace $T$, 
then the characteristic polynomial of 
    the corresponding rational point $(K/\QQ,x)$ is 
\begin{equation} 
    f=t^3-t^2+\tfrac{1}{3}(1 - N)t+\tfrac{1}{27}(1+N(T-3)) \in \QQ[t]. 
\end{equation} 
Such a polynomial either has Galois group $C_3$ or 
    splits into three linear factors over $\QQ$, 
    with at most two linear factors being the same. 
Conversely, a monic trace-one polynomial $f=t^3-t^2+at+b \in \QQ[t]$ 
which either has Galois group $C_3$ or 
    splits into three linear factors over $\QQ$, 
    with at most two linear factors being the same, 
    can be expressed in this way 
    for precisely two rational points of $T$ 
    if $f$ has no repeated roots, or 
    for precisely one rational point of $T$ 
    if $f$ has a double root which is not a triple root. 
The elements $u+v\zeta \in \QQ(\sqrt{-3})^\times$ 
corresponding to $f$ 
will be the roots of the quadratic polynomial 
\begin{equation} 
    g = t^2 - \left(3-\frac{1-27b}{1-3a}\right)t + 1-3a \in \QQ[t]. 
\end{equation} 
The polynomial $f$ will have integral coefficients 
if and only if 
\begin{equation}\label{eqn:integralityConditionsUsingCyclo} 
    \begin{cases}
        u^2+v^2-uv \in 1+3 \ZZ \text{ and }\\
        (u^2+v^2-uv)(3-2u+v) \in 1+27 \ZZ.
    \end{cases}
\end{equation} 
\end{proposition}

\begin{proof} 
The character lattice of a torus over $\QQ$ 
    as a Galois representation determines 
    the torus as an algebraic group up to isomorphism, 
    cf.~e.g.~\cite[Theorem~12.23]{milne}. 
Equation \eqref{eqn:formulaForNormCharacter} 
    below identifies the toric height 
    with the square-root of the norm. 
The formulas for $a$ and $b$ follow from expressing 
    $a$ and $b$ in terms of characters of $T$ 
    and then using \eqref{eqn:isoCocharLattices} 
    to reexpress these using characters on $\mathcal G$. 
\end{proof} 

\begin{figure}[h!] 
\scriptsize
\renewcommand{\arraystretch}{1.3}
\begin{tabular}{|l|l|l|l|c|}
    \hline
Cubic $f$ & Quadratic $g$ & $\mathrm{disc}(f)$ & $\mathrm{disc}(g)$ & $H(f)^2$ \\ \hline
$t^{3} - t^{2}$ & $t^{2} - 2 t + 1$ & $0$ & $0$ & $1$ \\
$t^{3} - t^{2} - t + 1$ & $t^{2} + 4 t + 4$ & $0$ & $0$ & $4$ \\
$t^{3} - t^{2} - 2 t + 1$ & $t^{2} + t + 7$ & $7^{2}$ & $-1 \cdot 3^{3}$ & $7$ \\
$t^{3} - t^{2} - 2 t$ & $t^{2} - \frac{20}{7} t + 7$ & $2^{2} \cdot 3^{2}$ & $-1 \cdot 2^{2} \cdot 3^{5} \cdot 7^{-2}$ & $7$ \\
$t^{3} - t^{2} - 4 t + 4$ & $t^{2} + \frac{70}{13} t + 13$ & $2^{4} \cdot 3^{2}$ & $-1 \cdot 2^{4} \cdot 3^{5} \cdot 13^{-2}$ & $13$ \\
$t^{3} - t^{2} - 4 t - 1$ & $t^{2} - 5 t + 13$ & $13^{2}$ & $-1 \cdot 3^{3}$ & $13$ \\
$t^{3} - t^{2} - 5 t - 3$ & $t^{2} - 8 t + 16$ & $0$ & $0$ & $16$ \\
$t^{3} - t^{2} - 6 t + 7$ & $t^{2} + 7 t + 19$ & $19^{2}$ & $-1 \cdot 3^{3}$ & $19$ \\
$t^{3} - t^{2} - 6 t$ & $t^{2} - \frac{56}{19} t + 19$ & $2^{2} \cdot 3^{2} \cdot 5^{2}$ & $-1 \cdot 2^{2} \cdot 3^{5} \cdot 5^{2} \cdot 19^{-2}$ & $19$ \\
$t^{3} - t^{2} - 8 t + 12$ & $t^{2} + 10 t + 25$ & $0$ & $0$ & $25$ \\
$\quad\vdots$&$\quad\vdots$&$\quad\vdots$&$\quad\vdots$&$\vdots$\\
$t^{3} - t^{2} - 190 t + 719$ & $t^{2} + 31 t + 571$ & $7^{2} \cdot 571^{2}$ & $-1 \cdot 3^{3} \cdot 7^{2}$ & $571$ \\
$t^{3} - t^{2} - 190 t - 800$ & $t^{2} - \frac{23312}{571} t + 571$ & $2^{2} \cdot 3^{2} \cdot 5^{2} \cdot 7^{2} \cdot 13^{2}$ & $-1 \cdot 2^{2} \cdot 3^{5} \cdot 5^{2} \cdot 7^{2} \cdot 13^{2} \cdot 571^{-2}$ & $571$ \\
$t^{3} - t^{2} - 192 t + 720$ & $t^{2} + \frac{17710}{577} t + 577$ & $2^{6} \cdot 3^{6} \cdot 19^{2}$ & $-1 \cdot 2^{6} \cdot 3^{9} \cdot 19^{2} \cdot 577^{-2}$ & $577$ \\
$t^{3} - t^{2} - 192 t - 171$ & $t^{2} - 11 t + 577$ & $3^{4} \cdot 577^{2}$ & $-1 \cdot 3^{7}$ & $577$ \\
$t^{3} - t^{2} - 196 t + 1124$ & $t^{2} + \frac{922}{19} t + 589$ & $2^{4} \cdot 31^{2}$ & $-1 \cdot 2^{4} \cdot 3^{3} \cdot 19^{-2}$ & $589$ \\
$t^{3} - t^{2} - 196 t + 1109$ & $t^{2} + \frac{1483}{31} t + 589$ & $7^{4} \cdot 19^{2}$ & $-1 \cdot 3^{3} \cdot 7^{4} \cdot 31^{-2}$ & $589$ \\
$t^{3} - t^{2} - 196 t + 539$ & $t^{2} + \frac{673}{31} t + 589$ & $7^{2} \cdot 19^{2} \cdot 37^{2}$ & $-1 \cdot 3^{3} \cdot 7^{2} \cdot 31^{-2} \cdot 37^{2}$ & $589$ \\
$t^{3} - t^{2} - 196 t + 349$ & $t^{2} + 13 t + 589$ & $3^{4} \cdot 19^{2} \cdot 31^{2}$ & $-1 \cdot 3^{7}$ & $589$ \\
$t^{3} - t^{2} - 196 t + 196$ & $t^{2} + \frac{3526}{589} t + 589$ & $2^{4} \cdot 3^{2} \cdot 5^{2} \cdot 7^{2} \cdot 13^{2}$ & $-1 \cdot 2^{4} \cdot 3^{5} \cdot 5^{2} \cdot 7^{2} \cdot 13^{2} \cdot 19^{-2} \cdot 31^{-2}$ & $589$ \\
$t^{3} - t^{2} - 196 t - 704$ & $t^{2} - \frac{20774}{589} t + 589$ & $2^{4} \cdot 3^{6} \cdot 5^{2} \cdot 7^{2}$ & $-1 \cdot 2^{4} \cdot 3^{9} \cdot 5^{2} \cdot 7^{2} \cdot 19^{-2} \cdot 31^{-2}$ & $589$ \\
    \hline
\end{tabular}
    \caption{Some $f \in \ZZ[t]$ with Galois group $C_3$ and 
    the characteristic polynomials $g \in \QQ[t]$ of 
    their corresponding elements in $\QQ(\sqrt{-3})$.} 
\end{figure} 
%




\section{Toric heights}\label{sec:heightFunctions} 

In this section we show that the toric height $H(-,1)$ 
of a $D_0$-integral point $(K,x)$ of $T$ 
in the sense of \cite{MR1369408} 
is equal to $H(f) = \sqrt{1-3a}$ 
where $f$ is the characteristic polynomial of $x$. 

\begin{definition}
Let $w$ be a place of $E$. 
For any $x \in T(E_w)$ 
    the function 
    $\chi \mapsto \mathrm{ord}_w(\chi(x))$ 
    on characters $\chi \in X^\ast(T_{E_w})$ 
    determines an element of $X_\ast(T_{E_w})_\RR$. 
Let $$n_w(x) \in X_\ast(T_E)_\RR$$ 
    be the cocharacter corresponding to this element 
    under the canonical isomorphism 
    $X_\ast(T_{E_w})_\RR \cong X_\ast(T_{E})_\RR$ 
    induced by base change of the split torus $T_E$ 
    along $E \to E_w$. 
\end{definition}

For any place $v$ of $\QQ_v$ let $K_v$ denote 
the maximal compact subgroup of $T(\QQ_v)$. 
Evaluating characters of $T_E$ on $\QQ_v$-points 
gives a canonical bijection 
$$T(\QQ_v)=\mathrm{Hom}_{\Gamma(w/v)}(M_E,E_w^\times)$$ 
where $w$ is any place of $E$ over $v$. 
When $v$ is finite, $K_v$ may be identified with 
the subset of $O_w^\times$-valued homomorphisms 
    $K_v = \mathrm{Hom}_{\Gamma(w/v)}(M_E,O_w^\times) \subset 
    T(\QQ_v)$. 

\begin{proposition}\label{prop:basicfacts} 
Let $w$ be a place of $E$ lying over a place $v$ of $\QQ$. 
There is an exact sequence 
\begin{equation} 
    1 \longrightarrow K_v 
    \longrightarrow T(\QQ_v) 
    \xlongrightarrow{n_w}
    X_\ast(T_{E})_\RR^{\Gamma(w/v)}.
\end{equation} 
If $w$ is infinite then $n_w$ is surjective, and 
    if $w$ is finite then the image of $n_w$ is 
    the lattice $X_\ast(T_{E})^{\Gamma(w/v)}$. 
\end{proposition} 


\begin{proof} 
\cite[(1.3), p.~449]{MR291099} nearly proves the claim 
    but at the ramified place $w$ over $v = 3$ 
    only ensures that the image of $n_w$ 
    is a finite index subgroup of 
    $X_\ast(T_{E})^{\Gamma(w/v)}$. 
To see that the image of $n_w$ is all of 
    $X_\ast(T_{E})^{\Gamma(w/v)}$ 
    recall that the cocharacter lattices 
    of $T_E$ and $\mathcal G_E$ are isomorphic 
    as Galois representations via \eqref{eqn:isoCocharLattices}. 
Since $T(\QQ_v)$ and $K_v$ are determined by 
    the dual modules $M_E$ and $M_{E}'$, 
    it suffices to show that $n_w$ is surjective 
    when defined relative to $\mathcal G$; 
in more detail, there is a diagram 
\begin{equation}
\begin{tikzcd}
    1\ar[r] &\mathrm{Hom}_{\Gamma(w/v)}(M_E',O_w^\times) \ar[r]\ar[d]
      & \mathrm{Hom}_{\Gamma(w/v)}(M_E',E_w^\times) \ar[r,"n_w"]\ar[d]
      & \mathrm{Hom}_{\Gamma(w/v)}(M_E',\ZZ) \ar[d] \\
    1\ar[r] &\mathrm{Hom}_{\Gamma(w/v)}(M_E,O_w^\times) \ar[r]
      & \mathrm{Hom}_{\Gamma(w/v)}(M_E,E_w^\times) \ar[r,"n_w"]
      & \mathrm{Hom}_{\Gamma(w/v)}(M_E,\ZZ) 
\end{tikzcd}
\end{equation} 
where the vertical arrows are isomorphisms of abelian groups 
induced by the transpose of 
    the $\Gamma$-isomorphism $N_E \to N_E'$, 
    and the homomorphisms $n_w$ 
    correspond to post-composing with $\mathrm{ord}_w$. 
The diagram commutes 
so surjectivity of the upper $n_w$ implies surjectivity of 
    the lower $n_w$. 

To see that the upper $n_w$ is surjective, 
    observe that the upper row of the diagram 
    is the $\Gamma(w/v)$-invariants of 
    the short exact sequence of $\Gamma(w/v)$-modules 
\begin{equation} 
\begin{tikzcd}
    1\ar[r] &\mathrm{Hom}(M_E',O_w^\times) \ar[r]
      & \mathrm{Hom}(M_E',E_w^\times) \ar[r]
      & \mathrm{Hom}(M_E',\ZZ) \ar[r] &0
\end{tikzcd}
\end{equation} 
    (here the exactness on the right follows from 
    $\mathrm{Ext}^1(M_E',O_w^\times) = 0$ since $M_E'$ is free); 
thus the upper row of the diagram 
    continues to the first cohomology group 
    $H^1(\Gamma(w/v),\mathrm{Hom}(M_E',O_w^\times))$. 
Now recall that the group of units 
    $U$ in the group algebra of $C_3$ is 
    $\GG_m \times R^E_\QQ \mathbb G_m$ 
    where the first projection is the augmentation character, 
    so the torus $\mathcal G$ is isomorphic to 
    $R^E_\QQ \mathbb G_m$. 
This implies that $M_E'$ is a free $\ZZ \Gamma(w/v)$-module, 
$\mathrm{Hom}(M_E',O_w^\times)$ is coinduced, 
and therefore 
$H^1(\Gamma(w/v),\mathrm{Hom}(M_E',O_w^\times)) = 0$ 
so $n_w$ is surjective. 
\end{proof} 

The toric variety $S$ has at worst cyclic quotient singularities 
    since its fan is simplicial 
    so every Weil divisor on $S$ is $\QQ$-Cartier. 
The toric height 
    with respect to a Weil divisor $D$ for which $nD$ is Cartier 
    is defined as $H(-,\mathcal O(D))$ as $H(-,\mathcal O(nD))^{1/n}$. 
Let $D_0,D_1,D_2$ be 
    the three irreducible $T$-stable divisors 
    corresponding respectively 
    to the three generators $v_0,v_1,v_2$ 
    in $\Sigma$ 
    of the fan of $S$ (cf.~\cite[\S3.1]{MR1234037}). 
We call any formal $\mathbb C$-linear combination 
    $s_0D_0+s_1D_1+s_2D_2$ 
    a \textit{toric divisor} of $S$. 
A \textit{support function} is 
    a continuous $\Gamma$-invariant function 
    $\varphi \colon N_{E,\RR} \to \CC$ 
    whose restriction to any cone of $\Sigma$ is linear. 
Support functions and $\Gamma$-invariant toric divisors 
    are in bijection under 
    $$\varphi \leftrightarrow 
    (s_0,s_1,s_2)=(-\varphi(v_0),-\varphi(v_1),-\varphi(v_2))$$ 
    where $s_1 = s_2$ to ensure $\Gamma$-invariance. 
Any Cartier toric divisor $\sum_{e}s_eD_e$ corresponds to 
    a $T_E$-linearized line bundle 
    $\mathcal O(\sum_{e}s_eD_e)$ 
    whose corresponding support function $\varphi$ 
    satisfies 
    $\varphi(e) = -s_e$ for each $e \in {\Sigma}(1)$. 

\begin{definition}[{\cite{MR1369408}}]
    For $x = (x_w)_w \in T(\AA_E)$ 
    and $\varphi$ a support function let 
\begin{equation} 
    H(x,\varphi)
    =\prod_{w \in Pl_E}
    \left(q_w^{-\varphi(n_w(x_w))}\right)^{\frac{1}{[E:\QQ]}}
\end{equation} 
where $\varphi(n_w(x_w))$ is evaluated using 
    the canonical isomorphism $X_\ast(T_E) \cong X_\ast(T_{E_w})$. 
\end{definition}


The following simplified form is often useful. 
If $x = (x_v)_v \in T(\AA)$, embedded diagonally in $T(\AA_E)$, 
    then 
    the quantity $\varphi(n_w(x_v))$ is independent 
    of the choice of $w$ over $v$, and 
\begin{equation}\label{eqn:heightforrationaladeles} 
        H(x,\varphi)
        = \prod_{v \in M_\QQ} 
        q_v^{-\frac{1}{e_v}\varphi(n_w(x_v)) }
\end{equation} 
where $e_v$ is the ramification index of any prime of $E$ 
lying over $v$ ($1$ by definition if $v = \infty$). 


\subsection{Computing the local toric height} 

Let $L$ be a globally generated line bundle on $S$ 
and let $\{v_1,\ldots,v_N\} \subset H^0(S_E,L)$ 
    be a generating set of global sections. 
The standard height function on $S$ associated to $L$ 
and the generating set $\{v_1,\ldots,v_N\}$ is 
\begin{equation} 
    H(x,L,(v_i)_{i=1}^N) = 
    \prod_{w \in Pl_E}
    \max\left(
        \left|\frac{v_1(x)}{s(x)}\right|_w,
        \ldots, 
        \left|\frac{v_N(x)}{s(x)}\right|_w
        \right)^{\frac 1{[E:\QQ]}}
        \qquad(x \in S(E)).
\end{equation} 
where $s$ is any local nonvanishing section at $x$, 
and $|\cdot|_w = q_w^{-\mathrm{ord}_w(\cdot)}$ 
    if $w$ is nonarchimedean and 
    $|\cdot|_w=|\cdot|^{d_w}$ otherwise. 
The quantity $H(x,L,(v_i)_{i=1}^N)$ does not depend on 
the local section $s$ 
or the choice of splitting field.

If the line bundle $L$ is linearized by the open torus $T$ of $S$ 
in the sense of \cite[\S1.3]{MR1304906}, 
then the space of sections of $L$ on any $T$-stable 
open subset of $S$ carries a linear action of $T$ 
and may therefore be diagonalized. 
The \emph{toric height} on $S$ associated to 
    a $T$-line bundle $L$ 
    is the standard height function on $S$ 
    defined using a basis of \emph{weight vectors} for $H^0(S,L)$. 
The advantage of this height is that 
    its local height functions 
    are amenable to harmonic analysis ---
    namely their Fourier transforms 
    have a simple form. 

The next lemma computes the weight vectors we need 
    to express the toric height relative 
    to the toric divisor $D_0$. 

\begin{lemma}\label{lemma:heightformula} 
Let $\mathbf 1$ denote the canonical nowhere-vanishing 
    global section in $H^0(S_E,\mathcal O(3D_0))$. 
The space $H^0(S_E,\mathcal O(3D_0))$ 
    is spanned over $E$ by the following four weight vectors: 
\begin{equation}\label{eqn:weightVectors} 
    \mathbf 1,\quad (1-3e_2e_1^{-2})\mathbf 1,\quad 
    (e_1^3-\tfrac 92 e_1 e_2+ \tfrac{27}{2} e_3 
        + \tfrac{\sqrt{-27}}{2} \sqrt{\mathrm{disc}})
        e_1^{-3}\mathbf 1,
        \quad
    (e_1^3-\tfrac 92 e_1 e_2+ \tfrac{27}{2} e_3 
        - \tfrac{\sqrt{-27}}{2} \sqrt{\mathrm{disc}})
        e_1^{-3}\mathbf 1
\end{equation} 
    where $\sqrt{\mathrm{disc}} = (X-Z)(Y-X)(Z-Y)$. 
The associated characters of $T_E$ are, respectively, 
\begin{equation} 
    1,\quad \chi^{a'+b'},\quad \chi^{2a'+b'},\quad\chi^{2b'+a'}
\end{equation} 
where $a' = 2v_1^\vee - v_2^\vee$ and $b' = 2v_2^\vee - v_1^\vee$ 
    in the character lattice $M_E = X^\ast T_E$ 
    and $(v_1^\vee,v_2^\vee)$ 
    is the dual basis to $(v_1,v_2)$. 
\end{lemma} 

\begin{proof} 
Let $\varphi_0$ be the support function corresponding to $-D_0$. 
\FloatBarrier 
\begin{figure}[h!] 
\begin{tikzpicture}[scale=2] 
\filldraw [black] (0,0) circle (.1em);
%

\filldraw [black] ({-sqrt(3)/2},-1/2) circle (.1em); 
\filldraw [black] ({sqrt(3)/2},-1/2) circle (.1em); 
\filldraw [black] (0,1) circle (.1em); 

\node[yshift=1.5em] at ({1/sqrt(3)},0) {$\varphi_0 = 0$};
\node[yshift=1.5em] at ({-1/sqrt(3)},0) {$\varphi_0 = v_1^\vee$};
\node[yshift=.5em] at (0,{-1/sqrt(3)}) {$\varphi_0 = v_2^\vee$};
 
\draw (0,0) -- ({sqrt(3)/2},-1/2);
\draw (0,0) -- (0,1);
\draw (0,0) -- ({-sqrt(3)/2},-1/2);

\node[yshift=1em] at ({sqrt(3)/2},-1/2) {$v_1$};
\node[yshift=1em] at (0,1) {$v_2$};
\node[yshift=1em] at (-{sqrt(3)/2},-1/2) {$v_0$};


\end{tikzpicture} 
    \caption{The support function $\varphi_0$ on $N_{E,\RR}$.}
\end{figure}
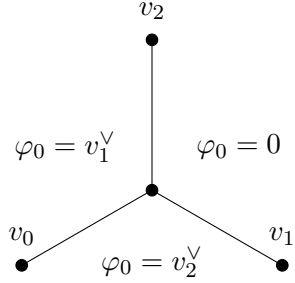 
\FloatBarrier 
On $S_E = S \otimes E$ we have 
    the weight decomposition 
    \cite[p.~66, \S3.4]{MR1234037} 
\begin{equation} 
H^0(S_E,\mathcal O(3D_0))
    \cong
    \bigoplus_{u \in P \cap M_E}
    E \chi^u
\end{equation} 
where $P$ is the polyhedron in $M_{E,\mathbb R}$ 
defined by 
$$
P = \{u \in M_{E,\mathbb R}:
u \geq 3\varphi_0 \text{ on $N_{E,\mathbb R}$}
\}.
$$ 
Let $(v_1^\vee,v_2^\vee) \in M_{E,\QQ}$ 
    be the dual basis to $(v_1,v_2) \in N_E$. 
Write $u = u_1v_1^\vee + u_2v_2^\vee \in M_E$. 
The polyhedron $P$ is cut out by the inequalities 
\begin{equation} 
    \begin{cases}
        u \geq 0 \,\,\text{on $\sigma_{12}$}\\
        u \geq 3v_2^\vee \,\,\text{on $\sigma_{10}$}\\
        u \geq 3v_1^\vee \,\,\text{on $\sigma_{20}$}.
    \end{cases}
\end{equation} 
Figure~\ref{fig:polyhedron} depicts the polyhedron $P$ 
    when $N_E$ is identified with 
    the lattice in $\mathbb R^2$ generated by 
    $\omega=(1,0)$ and 
    $\gamma \omega = (\frac12,\frac{\sqrt{3}}{2})$. 
Then the character lattice $M_E$ is generated by 
    $a'=(1,-\frac{\sqrt{3}}{3})$ and $b'=(0,\frac{2\sqrt{3}}{3})$. 
We have that $u_1 = \frac32 x - \frac{\sqrt{3}}{2}y$ and 
    $u_2 = \sqrt{3} y$ 
    where $x,y$ are the standard coordinates on $\RR^2$, 
and the polyhedron $P$ 
is cut out by the inequalities 
\begin{equation} 
    \begin{cases}
        \frac32 x - \frac{\sqrt{3}}{2}y \geq 0\\
        {\sqrt{3}}y \geq 0\\
        3\geq\frac32 x + \frac{\sqrt{3}}{2}y.
    \end{cases}
\end{equation} 
\begin{figure}[h!] 
\begin{tikzpicture}[scale=2,rotate=30] 

\fill[gray, opacity=0.5] 
    (0,0) -- ({sqrt(3)/2},-1/2) -- ({sqrt(3)/2},1/2) -- cycle;

\filldraw (0,0) circle (.1em);
\filldraw ({1/sqrt(3)},0) circle (.1em);
\filldraw ({2/sqrt(3)},0) circle (.1em);
\filldraw ({-1/sqrt(3)},0) circle (.1em);
\filldraw ({-2/sqrt(3)},0) circle (.1em);
\filldraw ({sqrt(3)/2},1/2) circle (.1em); 

\filldraw ({1/(2*sqrt(3))},1/2) circle (.1em); 
\filldraw ({2/(2*sqrt(3))},1) circle (.1em);
\filldraw ({-1/(2*sqrt(3))},-1/2) circle (.1em);
\filldraw ({-2/(2*sqrt(3))},-1) circle (.1em);
\filldraw ({-1/(2*sqrt(3))},1/2) circle (.1em);
\filldraw ({-2/(2*sqrt(3))},1) circle (.1em);
\filldraw ({1/(2*sqrt(3))},-1/2) circle (.1em);
\filldraw ({2/(2*sqrt(3))},-1) circle (.1em);

\filldraw ({sqrt(3)/2},-1/2) circle (.1em);
\filldraw ({-sqrt(3)/2},-1/2) circle (.1em);
\filldraw (0,1) circle (.1em);

\filldraw ({-sqrt(3)/2},1/2) circle (.1em);
\filldraw (0,-1) circle (.1em);

\node[xshift=-1em] at ({1/(2*sqrt(3))},-1/2) {$a'$};
\node[xshift=-1em] at (0,0) {$0$};
\node[xshift=-1em] at ({1/(2*sqrt(3))},1/2) {$b'$};

\end{tikzpicture} 
    \\[1em] 
    \caption{The polyhedron $P \subset M_{E,\mathbb R}$ for the $T_E$-line bundle $\mathcal O(3D_0)$.}
    \label{fig:polyhedron}
\end{figure}
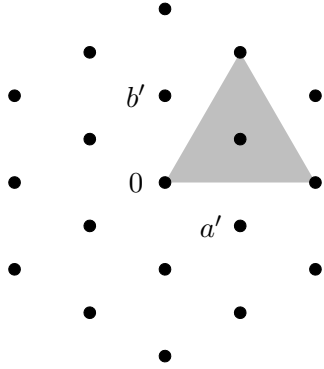
We conclude 
    that $h^0(S_E,\mathcal O(3D_0)) = 4$. 

The global section $\mathbf 1$ is clearly 
    the weight vector in $H^0(S_E,\mathcal O(3D_0))$ 
    with trivial $T_E$-action. 
We may find the other three weight vectors in 
    $H^0(S_E,\mathcal O(3D_0))$ 
    by twisting $\mathbf 1$ by 
    the three nontrivial characters in $P$. 
Using the formulas from \S\ref{sec:ToricDescription}, 
one finds that 
\begin{equation}\label{eqn:formulaForNormCharacter} 
    \chi^{a'+b'}(uv_0' + vv_1'+wv_2') = \frac{vw}{u^2} = 
    \frac{(X+\zeta Y+\zeta^2 Z)(X+\zeta^2 Y+\zeta Z)}{(X+Y+Z)^2}
    = \frac{e_1^2-3e_2}{e_1^2}
\end{equation} 
with associated weight vector $\chi^{a'+b'}\mathbf 1$. 
Similarly, 
\begin{equation} 
    \chi^{2b'+a'}(uv_0' + vv_1'+wv_2') 
    =\frac{w^3}{u^3}
    = \frac{e_1^3-\tfrac 92 e_1 e_2+ \tfrac{27}{2} e_3 
        - \tfrac{\sqrt{-27}}{2} \sqrt{\mathrm{disc}}}{e_1^3}
\end{equation} 
and $\chi^{2a'+b'} = \gamma \chi^{2b'+a'}$ 
    is the conjugate character. 
\end{proof} 



\subsection{Completing the orbit parametrization} 

Consider the function 
\begin{align} 
    T(\mathbb Q) &\to \mathbb P_3(E)\\
    (K/\mathbb Q,x) &\mapsto [w_1:w_2:w_3:w_4]
\end{align} 
where $w_1,\ldots,w_4$ are the weight vectors in 
$H^0(S_E,\mathcal O(3D_0))$ given by \eqref{eqn:weightVectors}. 

\begin{proposition}\label{prop:canonicalHeight}
The characteristic polynomial $f = t^3-t^2+at+b \in \QQ[t]$ 
    of a rational point 
    $(K/\mathbb Q,x) \in T(\mathbb Q)$ 
    has integer coefficients if and only if 
    $(K/\mathbb Q,x)$ is $D_0$-integral. 
For any $D_0$-integral rational point $(K/\QQ,x)$ on $T$, 
\begin{equation} 
    H((K/\mathbb Q,x),\mathcal O(D_0)) = 
    H(f)=
    \sqrt{1-3a}.
\end{equation} 
\end{proposition}

\begin{proof} 
First we verify that 
    the $C_3$-invariant functions 
    $e_1,e_2,e_3,\sqrt{\mathrm{disc}}$ of $X,Y,Z$ 
    appearing in the formulas \eqref{eqn:weightVectors} 
    for the weight vectors are polynomial functions 
    of the coefficients of the characteristic polynomial of $x$. 
By \cite[Prop.~2.5]{odesky2023moduli} 
    the unit 
\begin{equation} 
u = 
\sum_{g \in C_3}
g(x)[g^{-1}] \in \mathcal G(K)
\end{equation} 
maps to $(K/\mathbb Q,x)$ under 
$\mathcal G \to \mathcal G/C_3$. 
Thus the three rational functions 
    $X/\varepsilon,Y/\varepsilon,Z/\varepsilon$ 
on $\mathbb P$ evaluate on $u$ 
    to the Galois conjugates of $x$, 
    and therefore any $C_3$-invariant polynomial 
    in $X/\varepsilon,Y/\varepsilon,Z/\varepsilon$ 
    is a polynomial function in the 
    coefficients of the characteristic polynomial of $x$. 

This proves the `if' direction of the first assertion, 
    since $a$ and $b$ 
    are the values at $(K/\QQ,x)$ of 
    the $C_3$-invariant polynomials 
$e_2(X/\varepsilon,Y/\varepsilon,1-X/\varepsilon-Y/\varepsilon)$
and 
$-e_3(X/\varepsilon,Y/\varepsilon,1-X/\varepsilon-Y/\varepsilon)$ 
    in $\ZZ[X/\varepsilon,Y/\varepsilon]^{C_3}$. 
For the `only if' direction, 
    first we use that 
\begin{equation}\label{eqn:C3invariants} 
    \ZZ[X,Y,Z]^{C_3} = \ZZ[e_1,e_2,e_3,X^2Y+Y^2Z+Z^2X] 
\end{equation} 
    (see~e.g.~\cite[Example~4.6]{campbell2023permutation}). 
For any integer $d \geq 1$, 
    dehomogenizing with respect to $\varepsilon$ 
    induces an isomorphism of $C_3$-modules 
    $\ZZ[X,Y,Z]_d \cong \ZZ[X/\varepsilon,Y/\varepsilon]_{\leq d}$ 
    where $(-)_d$ (resp.~$(-)_{\leq d}$) 
    denotes the submodule of homogeneous 
    degree $d$ elements 
    (resp.~degree $\leq d$ elements). 
    In particular, 
    $$\ZZ[X,Y,Z]_d^{C_3} 
    \cong \ZZ[X/\varepsilon,Y/\varepsilon]_{\leq d}^{C_3},$$ 
    and so $(K/\QQ,x)$ is $D_0$-integral 
    if and only if the four generators of $\ZZ[X,Y,Z]^{C_3}$ 
    are integral on $(K/\QQ,x)$. 
In fact, it already suffices for $e_2$ and $e_3$ to be integral: 
if $e_2$ and $e_3$ evaluate to integers on $(K/\QQ,x)$, 
    then 
    $X^2Y$ will evaluate to an integral element of $K$ 
    and its trace will be an integer, 
    equal to the value of the last generator. 
This proves the first assertion. 

To compute the toric height, 
    we use \cite[p.~68]{MR1234037} 
    to express the support function $\varphi_0$ 
    associated to $D_0$ 
    using the weight vectors in $H^0(S_E,\mathcal O(3D_0))$ 
    found in Lemma~\ref{lemma:heightformula}. 
The local toric height $H_v$ with respect to $\mathcal O(3D_0)$ 
    of any point $(K,x) \in T(\QQ)$ is 
\begin{align}\label{eqn:localHeight} 
    &\max\left(
        \left|\frac{w_1(x)}{\mathbf 1(x)}\right|_w,
        \ldots,
        \left|\frac{w_4(x)}{\mathbf 1(x)}\right|_w
        \right)^{\frac{1}{[E:\QQ]}}\\ \nonumber
    &=\max\left(
        1,
        |1-3e_2|_w,
        \left| 1-\tfrac{9}{2} e_2+ \tfrac{27}{2} e_3 
        - \tfrac{\sqrt{-27}}{2} \sqrt{\mathrm{disc}} \right|_w, 
        \left| 1-\tfrac{9}{2} e_2+ \tfrac{27}{2} e_3 
        + \tfrac{\sqrt{-27}}{2} \sqrt{\mathrm{disc}} \right|_w
        \right)^{\frac{1}{[E:\QQ]}}
\end{align} 
    where $|\cdot|_w = q_w^{-\mathrm{ord}_w(\cdot)}$ 
    if $w$ is nonarchimedean and $|\cdot|_w=|\cdot|^{d_w}$ otherwise. 
When $(K,x)$ is $D_0$-integral, 
    the only contribution 
    to the height is the local contribution 
    from the complex place $w$ of $E$ at infinity, 
    which is 
\begin{align}\label{eqn:heightExtraFactor} 
    \max\left(
        1,
        |1-3e_2|^2,
        \left|1-\tfrac 92 e_2+ \tfrac{27}{2} e_3 
        + \tfrac{\sqrt{-27}}{2} \sqrt{\mathrm{disc}}\right|^2
        \right)^{1/2}. 
\end{align} 
A short computation shows that 
\begin{equation} 
\left|1-\tfrac 92 e_2+ \tfrac{27}{2} e_3 
        + \tfrac{\sqrt{-27}}{2} \sqrt{\mathrm{disc}}\right|^2
        =
        (1-3e_2)^3.
\end{equation} 
Thus $1-3e_2 > 0$ and 
    $(1-3e_2)^3 \geq (1-3e_2)^2$ 
    which shows that 
\begin{equation} 
    H((K,x),\mathcal O(D_0)) = 
    H((K,x),\mathcal O(3D_0))^{1/3}=
    \sqrt{1-3e_2} = H(f). \qedhere
\end{equation} 
\end{proof} 

\begin{rem}\label{rmk:rootVsToric}
As a function of characteristic polynomials 
    $t^3-t^2+at+b$ of rational points on $T$, 
    the quotient 
\begin{equation} 
    \frac{\sqrt{3}\max(|a|^{1/2},|b|^{1/3})}{\sqrt{1-3a}}
\end{equation} 
is bounded and tends to $1$ as $a,b \to \infty$. 
This shows that the toric height is equivalent 
    to the ``root height'' in 
    Theorem~\ref{thm:asymptoticPolynomialCount}. 
\end{rem}



\section{The Poisson summation formula} 

In this section we prove the following formula 
for the height zeta function for $D_0$-integral 
rational points on the open torus of $S$. 

\begin{theorem}\label{thm:heightZetaFunctionFormula2}
Fix any $s \in (\CC^{\Sigma(1)})^{\Gamma}$ 
    with $\mathrm{Re}(s_e) \gg 0$ for every $e \in \Sigma_w(1)$. 
Then the multivariate Dirichlet series 
\begin{equation} 
    Z(s)=
    \sum_{\substack{P \in T(\mathbb Q)\\\text{$D_0$-\emph{integral}}}}
        H(P,s)^{-1}
\end{equation} 
is absolutely convergent and equals 
\begin{align}
    \left({1-3^{-z}}\right)
    \zeta(z)
    \prod_{q\equiv2\pmod{3}}
        \bparen{1+\frac{1}{q^{z}}}^{-1} 
    \prod_{p\equiv1\pmod{3}}
        \bparen{1+\frac{3}{p^z}\bparen{1-\frac{1}{p^z}}^{-1}}
\end{align}
    where $z = \frac{1}{2}(s_0+s_1+s_2)$. 
This multivariate Dirichlet series admits 
    a meromorphic continuation to 
    $\{s \in (\CC^{\Sigma(1)})^{\Gamma} : 
    \mathrm{Re}(s_0+s_1+s_2) > 1\}$. 
\end{theorem}

For the proof, 
we recall some well-known facts from harmonic analysis. 
For any finite place $v$ of $\mathbb Q$ 
let $d^\times x_v$ be the Haar measure on $T(\mathbb Q_v)$ 
for which the maximal compact subgroup has measure one, 
    and at the infinite place choose the Haar measure 
    $d^\times x_\infty$ on $T(\mathbb R)$ 
    for which $\ZZ v_0 \subset N_{\RR}$ 
    is a unimodular lattice 
    with respect to the pushforward to $N_\RR$ under $n_w$ 
    of $d^\times x_\infty$. 
For any finite set $S$ of places of $\QQ$ 
    containing $v=\infty$ let $\AA_{S}$ 
    denote the subring of adeles which are integral 
    at places not in $S$. 
There is a unique Haar measure on $T(\mathbb A)$, denoted $d^\times x$, 
    whose restriction to $T(\AA_{S}) = \prod_{v \in S} T(\QQ_v) \times \prod_{v \not \in S} K_v$ 
    is the product measure $\prod_{v \not \in S} d^\times x_v$ 
    for all $S$. 
The Fourier transform of any factorizable integrable function 
$f = \otimes_v f_v \in L^1(T(\mathbb A))$ is defined by 
\begin{equation} 
    \widehat{f}(\chi) = 
    \int_{T(\AA)}
    f(x)\chi(x)^{-1}\,d^\times x
    =
    \prod_v
    \int_{T(\QQ_v)}
    f_v(x)\chi_v(x)^{-1}\,d^\times x_v.
\end{equation} 
The subgroup $E^\times =T(\mathbb Q)$ is discrete in 
    $\AA_E^\times = T(\AA)$. 
We equip $T(\mathbb Q)$ with its counting measure 
    and the quotient group $T(\mathbb Q)\backslash T(\mathbb A)$ 
    with the quotient measure (also denoted $d^\times x$) 
    of $d^\times x$ by the counting measure. 
The \textit{dual measure} $d\chi$ of this quotient measure 
    is by definition the unique Haar measure on 
    $(T(\QQ)\backslash T(\AA))^\vee$ 
    with the property that 
    for all $F \in L^1(T(\QQ)\backslash T(\AA))$ satisfying 
    $\widehat{F} \in L^1((T(\QQ)\backslash T(\AA))^\vee)$, 
    the Fourier inversion formula holds: 
\begin{equation} 
    F(x) = 
    \int_{(T(\QQ)\backslash T(\AA))^\vee}
        \widehat{F}(\chi)\chi(x)\,d\chi.
\end{equation} 
Let $T(\mathbb Q)^\perp$ denote the 
    the subgroup of characters on $T(\AA)$ that are trivial on 
    $T(\mathbb Q)$; 
    this subgroup is canonically isomorphic to 
    $(T(\QQ)\backslash T(\AA))^\vee$. 
Let $f \in L^1(T(\AA))$. 
The general Poisson summation formula --- 
following from the classical proof for $\ZZ \subset \RR$ --- 
says that if 
$\widehat{f}\,|_{T(\mathbb Q)^\perp} \in L^1(T(\mathbb Q)^\perp)$ 
then 
\begin{equation}\label{eqn:poissonae} 
    \int_{T(\QQ)} f(xy)\, dx = 
    \int_{T(\QQ)^\perp}
    \widehat{f}(\chi)
    \chi(y)
    \,d\chi 
\end{equation} 
for a.e.\ $y \in T(\mathbb Q)$ and 
suitably normalized Haar measure $d\chi$ on $T(\QQ)^\perp$ 
\cite[Theorem~4.4.2, p.~105]{MR1397028}. 


To apply the Poisson summation formula 
    we will compute the Fourier transform of 
\begin{equation} 
    x \mapsto H(x,-s,D_0) = H(x,-s) 1_{D_0}(x)
    \qquad(x \in T(\mathbb A))
\end{equation} 
where $1_{D_0} \colon T(\AA) \to \{0,1\}$ 
    is the characteristic function on $D_0$-integral points. 
The function $H(x,-s,D_0)$ is factorizable 
    so its Fourier transform is equal to 
    the product of the transforms of its local factors: 
\begin{equation} 
\widehat{H}(\chi,-s,D_0)
=
    \prod_{v\in M_\QQ}
    \widehat{H_v}(\chi_v,-s,D_0).
\end{equation} 
As usual, we say that 
a character $\chi$ on $T(\QQ_v)$ is \emph{ramified} 
if its restriction to the maximal compact subgroup is nontrivial, 
and otherwise it is \emph{unramified}. 

\begin{proposition}\label{prop:heightFourierTransform} 
Let $s \in (\CC^{\Sigma(1)})^{\Gamma}$ and 
    assume $\mathrm{Re}(s_e)>0$ for each $e \in\{0,1,2\}$. 
Let $w$ be the infinite place of $E$. 
Let $\chi\in T(\RR)^\vee$ be a unitary character. 
If $\chi$ is ramified then 
    $\widehat{H_\infty}(\chi,-s)$ is identically zero. 
If $\chi$ is unramified, then 
    $\chi(x)= e(\langle n_w(x) ,m\rangle)$ for all $x \in T(\RR)$ 
    for a unique $m \in M_{\RR}$, 
    and 
\begin{equation}\label{eqn:infiniteHeightFourierTransform} 
    \widehat{H_\infty}(m,-s) 
    =
    \left(\frac{-1}{2\pi i}\right)
    \frac{s_0+s_1+s_2}{2\pi i}
    \frac{1}{(m(v_0)+\frac{s_0}{2\pi i})(m(v_0)-\frac{s_1+s_2}{2\pi i})}.
\end{equation} 
Next let $v$ be a finite place of $\QQ$. 
For any unitary character $\chi\in T(\QQ_v)^\vee$, 
    the integral defining $\widehat{H_v}(\chi,-s,D_0)$ 
    converges absolutely to a holomorphic function of $s$ 
    in the region 
    $$\{s \in (\CC^{\Sigma(1)})^\Gamma : 
        \mathrm{Re}(s_1),\mathrm{Re}(s_2)>0 \}.$$ 
Assume $v \neq 3$. 
Let $w$ be any place of $E$ lying over $v$. 
The local characteristic function $1_{D_0,v}$ is $K_v$-invariant. 
If $\chi$ is ramified, then $\widehat{H_v}(\chi,-s,D_0)$ 
    is identically zero. 
If $\chi$ is unramified 
then we may regard $\chi$ as a character on 
    $X_\ast(T_{E})^{\Gamma(w/v)}$ 
    (Proposition~\ref{prop:basicfacts}) 
    and 
\begin{equation}\label{eqn:finiteHeightFourierTransform} 
\widehat{H_v}(\chi,-s,D_0)
    =
    \sum_{\substack{n \in X_\ast(T_{E})^{\Gamma(w/v)}\\ 
    n \in \mathbb R_{\geq 0}v_1+\mathbb R_{\geq 0}v_2} }
    \chi(n)^{-1} q_v^{\varphi(n)} . 
\end{equation} 
If $v = 3$ then 
    the support of $x \mapsto H_3(x,-s,D_0)$ is 
    the unique subgroup $K_{3,2}$ of $K_3$ of index six. 
Under the isomorphism $T(\QQ_3) \to E_3^\times$ the support 
    corresponds to the subgroup $1+3O_{E,w}$ of $O_{E,w}^\times$ 
    where $w$ is the unique place of $E$ lying over $3$. 
\end{proposition} 

\begin{rem} 
The local Fourier transforms --- 
    and therefore the entire Poisson summation argument --- 
    must be computed \emph{before} restricting to 
    the line in $\mathrm{Pic}^T(S) \otimes \CC$ 
    spanned by the $T_E$-line bundle $\mathcal O(D_0)$ of interest 
    since $x \mapsto H_v(x,-s,\mathcal O(D_0))$ will not 
    be integrable for any place $v \neq 3,\infty$ 
    once either of $s_1$ or $s_2$ vanishes, 
    no matter how large and positive $\mathrm{Re}(s_0)$ is. 
\end{rem} 

\begin{proof} 
Note that $1_{D_0,\infty}$ is identically one 
    since integrality conditions are only imposed at finite places, 
    and also observe that 
    the integrand is $K_\infty$-invariant. 
If $\chi$ is ramified then $\widehat{H_\infty}(\chi,-s,D_0)$ 
    vanishes by Schur's lemma, 
    so suppose $\chi$ is unramified. 
Then 
\begin{align} 
    \widehat{H_\infty}(\chi,-s,D_0)
    =
    \int_{T(\RR)}
    H_\infty(x,-s)1_{D_0,\infty}(x)\chi(x)^{-1}
    \, d^\times x_\infty 
    &=
\int_{N_{w,\RR}}
    H_\infty(y,-s) e(-\langle y,m\rangle)
    \,d\mu(y)\\
    &=
\int_{N_{w,\RR}}
    e^{\varphi(y)} e(-\langle y,m\rangle)
    \,d\mu(y).
\end{align} 
Next we compute that 
\begin{align} 
\int_{N_{w,\RR}}
    e^{\varphi(y)} e(-\langle y,m\rangle)
    \,d\mu(y)
    &=
\int_{\RR_{\geq 0}}
    e^{\varphi(yv_0)} e(-\langle yv_0,m\rangle)
    \,d\mu(y)
    +
\int_{\RR_{\geq 0}}
    e^{\varphi(-yv_0)} e(\langle yv_0,m\rangle)
    \,d\mu(y)\\
    &=
\int_{\RR_{\geq 0}}
    e^{-y(s_0+2\pi im(v_0))}
    \,d\mu(y)
    +
\int_{\RR_{\geq 0}}
    e^{-y(s_1+s_2-2\pi im(v_0))}
    \,d\mu(y)\\
    &=
    (s_0+2\pi im(v_0))^{-1}
    +
    (s_1+s_2-2\pi im(v_0))^{-1}\\
    &=
    \left(\frac{-1}{2\pi i}\right)
    \left(
    \left(-m(v_0)-\frac{s_0}{2\pi i}\right)^{-1}
    +
    \left(m(v_0)-\frac{s_1+s_2}{2\pi i}\right)^{-1}\right)\\
    &=
    \left(\frac{-1}{2\pi i}\right)
    \frac{-(s_0+s_1+s_2)}{2\pi i}
    \frac{1}{(-m(v_0)-\frac{s_0}{2\pi i})(m(v_0)-\frac{s_1+s_2}{2\pi i})}
\end{align} 
which proves the claimed formula. 

Next let $v$ be a finite place of $\QQ$ and 
let $w$ be any place of $E$ lying over $v$. 
Let $N_w = X_\ast(T_{E})^{\Gamma(w/v)}$. 
The weight vectors in $H^0(S_E,\mathcal O(3D_0))$  
    correspond to the characters 
    $0,v_1^\vee,v_2^\vee,3v_1^\vee,3v_2^\vee$ in $M_E$, 
    so from \eqref{eqn:localHeight} 
    we see that the local height $H_v(x,D_0)$ is $\leq 1$ 
    if and only if 
    $n_w(x) \in \mathbb R_{\geq 0}v_1+\mathbb R_{\geq 0}v_2$. 

Now consider the sub-$O_E$-module 
\begin{equation}\label{eqn:inclusion} 
O_E\langle w_1,w_2,w_3,w_4\rangle \subset 
    O_E[X,Y,Z]^{C_3}_3 = 
    O_E\langle e_1^3,e_1e_2,e_3,\delta\rangle 
\end{equation} 
where $\delta = X^2Y + Y^2Z+Z^2X$ (cf.~\eqref{eqn:C3invariants}). 
From the formulas for the weight vectors, 
    one computes that the homomorphism taking 
    the basis vectors $e_1^3,e_1e_2,e_3,\delta$ to 
    the weight vectors $w_1,w_2,w_3,w_4$, 
    respectively, 
    has the matrix 
\begin{equation} 
\begin{pmatrix} 
    1&1&1&1\\
    0&-3&-3(2+\zeta)&-3(2+\zeta^{2})\\
    0&0&9(2+\zeta)&9(2+\zeta^{2})\\
    0&0&3(1+2\zeta)&3(1+2\zeta^{2})\\
\end{pmatrix} 
\end{equation} 
which has determinant $243\sqrt{-3}$. 

Assume $v \neq 3$. 
The cokernel of \eqref{eqn:inclusion} is a $3$-group, 
    so this inclusion becomes 
    an isomorphism after tensoring with $\ZZ_v$. 
Thus $x \in T(\QQ_v)$ is $D_0$-integral 
    $\iff e_2(x),e_3(x),\delta(x) \in \ZZ_v\iff
    w_1(x),\ldots,w_4(x) \in O_E \otimes \ZZ_v
    \iff
    H_v(x,D_0)\leq 1
    \iff
    n_w(x) \in \mathbb R_{\geq 0}v_1+\mathbb R_{\geq 0}v_2$. 
This also shows that $1_{D_0,v}$ is $K_v$-invariant 
    since the $w$-adic size of each weight vector is 
    unchanged under the action of $K_v$. 
If $\chi$ is ramified 
    then the Fourier transform of 
    $H_v(x,s)^{-1}1_{D_0,v}(x)$ 
    vanishes by Schur's lemma, 
    so suppose $\chi$ is unramified at $v$. 
The integrand is $K_v$-invariant 
    and $d^\times x_v(K_v)=1$ so 
\begin{equation}\label{eqn:localfourierassumovern}
    \int_{T(\QQ_v)} H_v(x,s)^{-1}1_{D_0,v}(x)\chi(x)^{-1}\,d^\times x_v 
=\sum_{n \in N_w}
q_v^{\frac{1}{e_v}\varphi(n)}
1_{D_0,v}(n) 
    \chi(n)^{-1}
    =
    \sum_{\substack{n \in X_\ast(T_{E})^{\Gamma(w/v)}\\ 
    n \in \mathbb R_{\geq 0}v_1+\mathbb R_{\geq 0}v_2} }
    \chi(n)^{-1} q_v^{\varphi(n)} . 
\end{equation}

For $v = 3$ we use the integrality conditions 
    \eqref{eqn:integralityConditionsUsingCyclo} 
    rephrased in terms of cyclotomic numbers 
    from Proposition~\ref{prop:cyclotomicParametrization}, 
    which in this local context take the form 
\begin{equation}
    e_2(x),e_3(x),\delta(x) \in \ZZ_3\iff
    \begin{cases}
        u^2+v^2-uv \in 1+3 \ZZ_3 \text{ and }\\
        (u^2+v^2-uv)(3-2u+v) \in 1+27 \ZZ_3
    \end{cases}
\end{equation} 
where $x \leftrightarrow u+v\zeta \in \QQ(\zeta) \otimes \QQ_3$. 
These conditions imply that $u+v\zeta$ is 
    a $3$-adic unit, 
    so the support of $x \mapsto H_3(x,-s,D_0)$ 
    is contained in $K_3$. 
Suppose that $z=u+v\zeta\in K_3$ is in the support. 
Let $N=u^2+v^2-uv$ and $T=2u - v$. 
Define $n,\tau\in \ZZ_3$ by $N=(1+3n)^{-1}$, $N(3-T)=1+27\tau$. 
One easily sees from these equations that 
\begin{equation} 
    T-2=3n+O(3^2)
    \quad\text{and}\quad
    1-T+N=3^2n^2+O(3^3)
\end{equation} 
and therefore from 
the Newton polygon of the characteristic polynomial of $z$, 
\begin{equation} 
    t^2-Tt+N = (t-1)^2-(T-2)(t-1)+1-T+N,
\end{equation} 
one concludes that $z \in 1 + 3O_w$. 
Conversely if $z = 1 + 3x$ with $x \in O_w$ then 
clearly $N(z) \in 1 + 3 \ZZ_3$ while 
$N(3-T) = 1+9(N-T^2)-27NT = 1 + 9(-3+9n)+O(3^3) \in 1 + 27 \ZZ_3$. 
\end{proof} 

%

To compute the quantities arising in the Poisson summation formula, 
    we need to parameterize the continuous 
    part of the automorphic spectrum 
    of the torus $T$. 
For any $x=(x_w)_{w} \in T(\AA_E)$ 
    let 
    $$L(x)= 
    \tfrac{1}{2}\sum_{w \in Pl_E}
        n_w(x_w)\log q_w \in N_{E,\RR}.$$ 
We can give a simpler expression for $L$ 
    using the isomorphism 
    $T \cong R^E_\QQ\GG_m$. 
It is easy to check that 
\begin{equation}\label{eqn:LAsNorm} 
    L(x)(\mathrm{N}) = \log |\mathrm{N}(x)|_{\AA}
\end{equation} 
where $\mathrm{N} \colon \AA_E^\times \to \AA^\times$ 
is the norm character. 
The norm character generates 
    the rational character lattice $M_E^\Gamma$ 
    so $N_E^\Gamma$ is generated 
    by the unique $\Gamma$-invariant cocharacter in $N_E$ 
    which takes the norm character to $1$. 
Thus for any $x \in \AA_E^\times = T(\QQ)$, 
    $L(x) = \tfrac12 \log |\mathrm{N}(x)|_{\AA}(v_1+v_2)
     \in N_\RR.$

\begin{proposition}\label{prop:Lmap} 
There is an exact sequence 
\begin{equation} 
    1 \longrightarrow K/\mu
    \longrightarrow T(\mathbb Q)\backslash T(\AA)
    \xlongrightarrow{L}
    N_{\mathbb R}
    \longrightarrow 0
\end{equation} 
where $K$ is the maximal compact subgroup of $T(\AA)$ 
and $\mu = T(\QQ) \cap K$. 
\end{proposition} 

\begin{proof} 
From \eqref{eqn:LAsNorm} we see 
    the kernel of $L$ is the norm-one subgroup of 
    the id\`ele class group 
    $T(\mathbb Q)\backslash T(\AA)$ of $E$. 
The rank of the group of units is zero and 
    the class group is trivial 
    so the norm-one subgroup of the id\`ele class group is 
    generated by $K/\mu$. 
Finally $L$ is surjective since $n_w$ is already surjective 
    for the complex place $w$ of $E$ 
    (Proposition~\ref{prop:basicfacts}). 
\end{proof} 

\begin{lemma}\label{lemma:KprimeGroup}
Let $K' \subset K$ 
    denote the subgroup which fixes 
    the characteristic function $1_{D_0} = \otimes_v 1_{D_0,v}$ 
    for $D_0$-integral points in $T(\AA)$. 
Then $$K' = K_{3,2} \times \prod_{v \neq 3} K_v$$ 
where $K_{3,2} \subset K_3$ is the unique subgroup with index $6$. 
There is an exact sequence 
\begin{equation} 
    1 \longrightarrow K/(K' \cdot \mu)
    \longrightarrow T(\mathbb Q)\backslash T(\AA)/K'
    \xlongrightarrow{L}
    N_{\mathbb R}
    \longrightarrow 0.
\end{equation} 
Restriction to the connected component of 
    the identity in $T(\mathbb Q)\backslash T(\AA)/K'$ 
    gives a canonical splitting 
    $s \colon N_\RR \to T(\mathbb Q)\backslash T(\AA)/K'$ 
    of $L$, inducing the isomorphisms 
\begin{align} 
    T(\mathbb Q)\backslash T(\AA)/K'
     &\xlongrightarrow{\sim} T(\QQ)\backslash T(\QQ) K/K' \times N_\RR 
     \xlongrightarrow{\sim} K/(K' \cdot \mu)  \times N_\RR \\
     T(\QQ) x K' &\mapsto (T(\QQ)xs(L(x))^{-1}K',L(x)) 
\end{align} 
    where the second map is defined using 
    the natural isomorphism 
    $T(\QQ)\backslash T(\QQ) K/K' 
     \cong K/(K' \cdot \mu)$. 
\end{lemma}

\begin{proof} 
The equality $K' = K_{3,2} \times \prod_{v \neq 3} K_v$ 
    follows from $K_v$-invariance of 
    the local characteristic functions 
    $1_{D_0,v}$ when $v \neq 3$ 
    and the computation of the support when $v = 3$ 
    from Proposition~\ref{prop:heightFourierTransform}. 
The short exact sequence is obtained by 
    taking the quotient by $K'$ of the first two groups in 
    the short exact sequence of Proposition~\ref{prop:Lmap}. 
The group $K/(K' \cdot \mu)$ is finite so 
    the natural quotient map  
    $T(\mathbb Q)\backslash T(\AA)/K' \to 
    T(\mathbb Q)\backslash T(\AA)/K$ 
    identifies the connected component of the identity 
    of $T(\mathbb Q)\backslash T(\AA)/K'$ 
    with $T(\mathbb Q)\backslash T(\AA)/K$. 
Thus the restriction of $L$ to 
    the connected component of the identity 
    of $T(\mathbb Q)\backslash T(\AA)/K'$ 
    is an isomorphism onto $N_\RR$, 
    so its inverse gives 
    the canonical splitting map $s$. 
\end{proof} 

Now we may prove Theorem~\ref{thm:heightZetaFunctionFormula2}. 

\begin{proof}[Proof of Theorem~\ref{thm:heightZetaFunctionFormula2}] 

Let $1_{D_0} \colon T(\AA) \to \{0,1\}$ 
    be the characteristic function on $D_0$-integral points. 
By the definition of $D_0$-integrality, 
    $1_{D_0} = \otimes_{v}1_{D_0,v}$ is a factorizable function. 
Take $f = H(\cdot,-s)1_{D_0}$. 
To apply the Poisson formula 
    we verify that $f$ is in $L^1(T(\AA))$ 
    and the restriction of $\widehat{f}$ 
    is in $L^1(T(\QQ)^\perp)$. 
From \eqref{eqn:heightforrationaladeles} we have 
\begin{equation}
    H(x,-s)1_{D_0}
        = 
        \prod_{v \in M_\QQ} 
        1_{D_0,v}(x_v)
        q_v^{\frac{1}{e_v}\varphi(n_w(x_v))}, 
        \quad
        x=(x_v)_v \in T(\AA) \subset T(\AA_E).
\end{equation} 
For any finite set $S$ of places of $\QQ$ 
    containing $v=\infty$ let $\AA_{S}$ 
    denote the subring of adeles which are integral 
    at places not in $S$. 
The chain of inequalities 
    $$\int_{T(\AA)} f(x) \,d^\times x
    =
    \lim_{C \text{ compact}}
    \int_{C} f(x) \,d^\times x 
    \leq 
    \lim_{S \text{ finite}}
    \int_{T(\AA_{S})} f(x) \,d^\times x 
    \leq 
    \int_{T(\AA)} f(x) \,d^\times x$$ 
    in the limits of larger $C$ and $S$ 
    shows that 
$$\int_{T(\AA)} f(x) \,d^\times x
=
    \lim_{S}
\int_{T(\AA_{S})} f(x) \,d^\times x 
    \leq
    \lim_{S}
    \prod_{\substack{v \in S\\v \neq 3}}
    \int_{T(\QQ_v)} 1_{D_0,v}(x_w)q_v^{\varphi(n_w(x_w))} 
    \,d^\times x_v
$$ 
    (recall that $H_3(x,-s,D_0)$ is supported in $K_3$ 
    by Proposition~\ref{prop:heightFourierTransform}). 
Let $|{\cdot}|$ be any norm on $N_\RR$. 
There is a constant $\rho>0$ such that 
    for any finite place $v\neq 3$, any place $w$ of $E$ lying over $v$, 
    and $n \in N_{E}^{\Gamma(w/v)}$, 
$$
\left|1_{D_0,v}(n)q_v^{\varphi(n)}\right| 
\leq 
\begin{cases} 
    0&\text{if $n$ is not in $\mathbb R_{\geq 0}v_1+\mathbb R_{\geq 0}v_2$,}\\
    q_v^{-\rho |n| \min\{\mathrm{Re}(s_1),\mathrm{Re}(s_2)\}}&\text{otherwise.}
\end{cases} 
$$ 
    Set $t = \min\{\mathrm{Re}(s_1),\mathrm{Re}(s_2)\}$. 
Then for $v \neq 3$ we have 
\begin{equation} 
    |
    \int_{T(\QQ_v)} 1_{D_0,v}(x_w)
        q_v^{\varphi(n_w(x_w))} \,d^\times x_v|
    \leq 
    \sum_{n \in N_{E}^{\Gamma(w/v)} \cap ( \mathbb R_{\geq 0}v_1+\mathbb R_{\geq 0}v_2)}
        q_v^{-\rho t|{n}|}
    \ll 
    \left(1-q_v^{-\rho t}\right)^{-\mathrm{rk} N_{E}^{\Gamma(w/v)}} 
\end{equation} 
where the implied constant is independent of $v$. 
For $v = \infty$ we have already seen that 
    $x\mapsto H_\infty(x,-s,D)$ is integrable 
    once $\mathrm{Re}(s_e) >0$ for all $e \in \Sigma_\Gamma(1)$ 
    (Proposition~\ref{prop:heightFourierTransform}). 
Thus for any finite set of places $S$, 
\begin{equation} 
    |
    \int_{T(\AA_{S})} f(x) \,d^\times x |\ll 
    \prod_{\substack{v \in S\\v \neq \infty}}
        \left(1-q_v^{-\rho t}\right)^{-1}
    \leq
    \zeta(\rho t)
\end{equation} 
which is finite for $t > 1/\rho$. 
Taking the limit over $S$ shows $f$ is integrable. 

Next we prove that the restriction of $\widehat{f}$ 
    to $T(\QQ)^\perp\cong (T(\AA)/T(\QQ))^\vee$ 
    is integrable by evaluating the integral. 
By Schur's lemma, this function is supported on 
$(T(\QQ)\backslash T(\AA)/K')^\vee$ 
    where $K' \subset K$ is the subgroup which fixes 
    the characteristic function $1_{D_0}$. 
We will use the isomorphism in Lemma~\ref{lemma:KprimeGroup}  
    to perform the integral over the automorphic spectrum of $T$. 
Let $C$ denote the finite group $K/(K' \cdot \mu)$. 
For any $\chi \in (T(\QQ)\backslash T(\AA)/K)^\vee$ 
    there is a unique $m \in M_\RR$ such that 
    $\chi(x) = e(\langle m,L(x) \rangle)$ 
    for all $x \in T(\AA)$. 
    Set $t = m(v_0)$ for $m \in M_\RR$ 
    and let $\chi_t$ be the corresponding character. 
Any $K'$-unramified automorphic character of $T$ 
    is of the form $\psi \chi_t$ 
    for a unique $\psi \in C^\vee$ and $t \in \RR$. 
The Haar measure on $T(\RR)$ was chosen so that 
$\ZZ v_0 \subset N_\RR$ was unimodular 
for the pushforward measure to $N_\RR$, and so 
\begin{equation}\label{eqn:integralOverM} 
\int_{(T(\AA)/T(\QQ))^\vee} 
    \widehat{f}(\chi)
    \,d\chi 
    = 
    \kappa
    \sum_{\psi \in C^\vee}
    \int_{\RR} 
    \widehat{f}(\psi \chi_t)
    \,dt
\end{equation} 
where $\kappa$ is a positive constant 
yet to be determined. 
The local $v$-adic component $\psi_v \in T(\QQ_v)^\vee$ 
of $\psi$ is 
$T(\QQ_v) \to T(\mathbb Q)\backslash T(\AA)/K' 
\twoheadrightarrow C \xrightarrow{\psi} \CC^\times$. 
Because the group $N_\RR = T(\RR)/K_\infty$ 
    has no nontrivial finite quotients, 
    the local component $\psi_\infty$ is trivial, 
and the infinite factor of $\widehat{f}$ 
    is \eqref{eqn:infiniteHeightFourierTransform}. 
With the help of \eqref{eqn:finiteHeightFourierTransform} 
    we find the product over the finite factors 
    besides $v=3$ is 
\begin{equation}\label{eqn:firstEta} 
    \prod_{{v \neq 3,\infty}}
    \widehat{f_v}(\psi_v\chi_{t,v})
    =
    \prod_{{v \neq 3,\infty}}
    \sum_{\substack{m \in X_\ast(T_{E})^{\Gamma(w/v)}\\ 
    m \in \mathbb R_{\geq 0}v_1+\mathbb R_{\geq 0}v_2} }
    \xi_v(m)^{-1} 
    q_v^{\varphi(m)}
    =
    \sum_{\eta}
    \xi(\eta)^{-1} \eta^{-s}
\end{equation} 
where $\xi = \psi \chi_t$, 
    $\eta = (m_w)_{v} \in \prod_{v\neq 3,\infty} X_\ast(T_{E})^{\Gamma(w/v)}$ 
    satisfies certain conditions, 
    $\xi(\eta) \coloneqq 
    \prod_{v \neq 3,\infty}\xi_v(m_w)$, 
    and 
    $\eta^{-s} \coloneqq
    \prod_{v \neq 3,\infty} 
    q_v^{\varphi(m_w)}$. 
This is a multivariate Dirichlet series in $s_1$ and $s_2$ 
    with summands indexed by $\eta$ 
    which is absolutely convergent when $s_1$ and $s_2$ have 
    sufficiently large and positive real parts, 
so the integral in \eqref{eqn:integralOverM} 
    may be distributed into the sum over $\eta$. 
With the help of 
    \eqref{eqn:infiniteHeightFourierTransform}, 
    we see that \eqref{eqn:integralOverM} equals 
\begin{equation}\label{eqn:integralOverM2} 
\kappa
    \left(\frac{-1}{2\pi i}\right)
    \frac{s_0+s_1+s_2}{2\pi i}
\sum_{\eta}
    \eta^{-s}
\int_{\RR} 
    \frac{1}
    {(t+\frac{s_0}{2\pi i})(t-\frac{s_1+s_2}{2\pi i})}
\sum_{\psi \in C^\vee}
    \widehat{f_3}(\xi_3) 
    \xi(\eta)^{-1}
    \,dt.
\end{equation} 
Let $K_{3,2}$ denote the support of 
    the local characteristic function $1_{D_0,3}$ 
(cf.~Proposition~\ref{prop:heightFourierTransform}). 
Since $\chi_{t,3}$ is trivial on $K_3$, 
    we have $\widehat{f_3}(\xi_3) = \widehat{f_3}(\psi_3)$. 
Recall that $n_w \colon T(\QQ_v) \to X_\ast(T_{E})^{\Gamma(w/v)}$ 
is surjective for any finite $w$ (Proposition~\ref{prop:basicfacts}). 
Since $T(\QQ) \subset T(\QQ_v)$ is dense for any $v$ 
($E^\times$ is obviously dense in 
$E^\times_v = (E \otimes \QQ_v)^\times$), 
there is a $y_v \in T(\QQ) \subset T(\QQ_v)$ 
which is a $n_w$-preimage of $m_w$ 
where $\eta = (m_w)_v$. 
Then 
\begin{align}\label{eqn:threeComponent} 
\sum_{\psi \in C^\vee}
    \widehat{f_3}(\xi_3) 
    \xi(\eta)^{-1}
    &=
\sum_{\psi \in C^\vee}
    \int_{K_{3,2}}
    \psi_3(x_3)^{-1}
    \,d^\times x_3\,
    \xi(\eta)^{-1}\\\nonumber
    &=
    \chi_t(\eta)^{-1}
\sum_{\psi \in C^\vee}
    \int_{K_{3,2}}
    \psi_3(x_3)^{-1}
    \prod_{v \neq 3,\infty}
    \psi_v(y_v)^{-1}
    \,d^\times x_3. 
\end{align} 
Since $K_v = K_v'$ for all $v \neq 3,\infty$ 
(Lemma~\ref{lemma:KprimeGroup}), 
the $3$-adic projection map $\mathrm{pr}_3$ 
induces an isomorphism 
$C \xrightarrow{\sim} K_3/(K_{3,2}\cdot \mathrm{pr}_3(\mu))$. 
Let $k_\eta \in K_3/(K_{3,2}\cdot \mathrm{pr}_3(\mu))$ be the image of 
$\prod_{v \neq 3,\infty} y_v$ under 
$T(\QQ)\backslash T(\AA)/K' \to C \to K_3/(K_{3,2}\cdot \mathrm{pr}_3(\mu))$ 
so that $\prod_{v \neq 3,\infty}\psi_v(y_v) = \psi_3(k_\eta)$. 
Explicitly, $k_\eta$ is $\prod_{v \neq 3,\infty} k_v$ 
where $k_v = (k_{v,v'})_{v'} \in K$ is the id\`ele with components 
\begin{equation} 
    k_{v,v'}
    =
\begin{cases} 
    1&\text{if $v' = v$,}\\
    y_v^{-1} |y_v|_v^{-1/2}&\text{if $v' = \infty$,}\\
    y_v^{-1}&\text{otherwise.}
\end{cases} 
\end{equation} 
In particular, $\psi_v(y_v) = \psi_3(\mathrm{pr}_3(y_v)^{-1})$. 

We claim that $\mathrm{pr}_3(y_v) \in K_{3,2}$ 
for all $v \neq 3,\infty$ 
(a priori it is only in $K_3$). 
This amounts to the assertion that 
every prime ideal in $O_E$ not dividing $3$ 
admits a generator that is congruent to $1 \pmod {3O_E}$. 
In other words, we claim that 
the ray class group $C_{\mathfrak m}$ 
of $O_E$ with modulus $\mathfrak m = 3 O_E$ is trivial. 
This follows from the short exact sequence 
\cite[Ch.~V, Theorem~1.7, p.~146]{milneCFT} 
(with notation defined there) 
\begin{equation} 
0 \longrightarrow
O_E^\times/O_{E,1}^\times\longrightarrow
    E_{\mathfrak m}^\times/E_{\mathfrak m,1}^\times\longrightarrow
    C_{\mathfrak m} \longrightarrow 
    C \longrightarrow 0 
\end{equation} 
which implies that 
\begin{equation} 
    h_{\mathfrak m} = h \cdot 
    \#(O_E^\times/O_{E,1}^\times)^{-1} \cdot 
    2^{r_0}\cdot
    N(\mathfrak m_0)
    \cdot
    \prod_{\mathfrak p | \mathfrak m_0}
    \left(1-N(\mathfrak p)^{-1}\right)
    = 1 \cdot 1^{-1} \cdot 2^0 \cdot 3^2\cdot (1-3^{-1}) = 1.  
\end{equation} 
Thus the integral in \eqref{eqn:threeComponent} 
simplifies down to 
\begin{equation} 
\sum_{\psi \in C^\vee}
    \int_{K_{3,2}}
    \psi_3(x_3)^{-1}
    \prod_{v \neq 3,\infty}
    \psi_v(y_v)^{-1}
    \,d^\times x_3. 
=
\sum_{\psi \in C^\vee}
    \int_{K_{3,2}}
    \psi_3(x_3k_\eta)^{-1}
    \,d^\times x_3 
    =
\sum_{\psi \in C^\vee}
    \int_{K_{3,2}}
    \psi_3(x_3)^{-1}
    \,d^\times x_3
\end{equation} 
by absorbing $k_\eta$ into the Haar measure. 
Since $K_{3,2} \subset \ker \psi_3$ 
for any $\psi \in C^\vee$, 
and recalling that $d^\times x_3(K_3) = 1$, 
this is equal to 
\begin{equation} 
\sum_{\psi \in C^\vee}
    \int_{K_{3,2}}
    \psi_3(x_3)^{-1}
    \,d^\times x_3
    =|C|\cdot d^\times x_3(K_{3,2}) = 
    |C| \cdot [K_{3}:K_{3,2}]^{-1} =  
    [K_{3,2}\mathrm{pr}_3(\mu):K_{3,2}] = 6. 
\end{equation} 

Returning to \eqref{eqn:integralOverM2}, 
we see that $Z(s)$ is equal to 
\begin{equation} 
\int_{(T(\AA)/T(\QQ))^\vee} 
    \widehat{f}(\chi)
    \,d\chi =
6\kappa
    \left(\frac{-1}{2\pi i}\right)
    \frac{s_0+s_1+s_2}{2\pi i}
\sum_{\eta}
    \eta^{-s}
\int_{\RR} 
    \frac{\chi_t(\eta)^{-1}\,dt}
    {(t+\frac{s_0}{2\pi i})(t-\frac{s_1+s_2}{2\pi i})}.
\end{equation} 
This can be evaluated using Cauchy's residue formula. 
The numerator of the integrand in \eqref{eqn:integralOverM2} 
is bounded in the upper half-plane 
and the denominator is $\ll t^{-2}$ 
so we may deform the path of integration along $\RR$ to 
the upper half-plane and obtain 
\begin{align} 
    \left(\frac{-1}{2\pi i}\right)
    ({s_0+s_1+s_2})
&\sum_{\eta}
    \eta^{-s}
    \mathrm{Res}\left[
        \frac{\chi_{t}(\eta)^{-1}}
    {(t-\frac{s_1+s_2}{2\pi i})}\,;\, t = \frac{-s_0}{2\pi i}
    \right]\\
=
&\sum_{\eta}
    \eta^{-s}
    \chi_{{\frac{s_0}{2\pi i}}}(\eta).
\end{align} 

We now describe the conditions determining $\eta = (m_w)_v$. 
A tuple $(m_w)_v \in \prod_v N_w$ corresponds to a summand 
of \eqref{eqn:firstEta} if and only if 
$m_w \in 
    \mathrm{im}\, n_w 
    \cap 
    (\mathbb R_{\geq 0}v_1+\mathbb R_{\geq 0}v_2)$ for all $w$. 
By Proposition~\ref{prop:basicfacts}, 
\begin{equation} 
    \mathrm{im}\, n_w = 
    \begin{cases}
        \ZZ \langle v_1,\omega \rangle&\text{if $w$ split,}\\
        N_E^\Gamma=\ZZ \langle v_0 \rangle &\text{otherwise.}
    \end{cases}
\end{equation} 
Any element of $N_E$ may be expressed as 
$a v_1 + b \omega 
= av_1 + b( \tfrac13(2v_1+v_2))
= (a+\tfrac23 b)v_1 + \tfrac13 bv_2$ 
for integers $a,b$. 
Then  
\begin{equation} 
    n_1 =
    \prod_{q\equiv2\pmod{3}}q^{c_q}
    \prod_{p\equiv1\pmod{3}}p^{a_p+\frac23 b_p}
\end{equation} 
and 
\begin{equation} 
    n_2 =
    \prod_{q\equiv2\pmod{3}}q^{c_q}
    \prod_{p\equiv1\pmod{3}}p^{\frac13 b_p}
\end{equation} 
for integer exponents $a_p,b_p,c_q$ almost all zero and satisfying 
\begin{equation} 
\begin{cases}
    a_p+\tfrac23b_p \text{ and } \tfrac13 b_p\geq0 
        &\text{if $p\equiv1\pmod{3}$,}\\
    c_q \geq 0 &\text{if $q\equiv2\pmod{3}$.}
\end{cases}
\end{equation} 

For a given $\eta = (m_w)_v$, 
let $n_1,n_2 \in \RR_{ \geq 1}$ be determined by the equality 
$$v_1 \log n_1 + v_2 \log n_2 
= \sum_{v\neq 3,\infty}m_w\log q_v.$$ 
The $v$-adic component of $\chi_t \in M_\RR$ ($t \in \RR$) 
is given by 
\begin{equation} 
    \chi_{t,v}(m_{w}) 
    = \chi_{t,v}(m_{w,1}v_1+m_{w,2}v_2) 
    = q_v^{-\pi i(m_{w,1}+m_{w,2})t}
\end{equation} 
and so 
    $$
    \chi_{t}(\eta) 
    =\prod_{v \neq 3,\infty}
        \chi_{t,v}(m_w)
    =\prod_{v \neq 3,\infty}
        q_v^{-\pi i (m_{w,1}+m_{w,2})t}
    =(n_1n_2)^{-\pi i t}.
    $$ 
We have that  
\begin{equation} 
    \eta^{-s} =
    \prod_{v \neq 3,\infty} q_v^{\varphi(m_w)}
    =(n_1 n_2)^{-s_1} 
\end{equation} 
and finally 
\begin{equation} 
\eta^{-s}
    \chi_{{\frac{s_0}{2\pi i}}}(\eta) 
    =(n_1n_2)^{-\left(\frac{s_0}{2}+s_1\right)}.
\end{equation} 
Set $z = \frac{s_0}{2}+s_1$ 
(the unique $M_\RR$-invariant linear form on 
$(\CC^{\Sigma(1)})^\Gamma$ up to scaling). 
Then 
\begin{equation}\label{eqn:zetaFunction1} 
\int_{(T(\AA)/T(\QQ))^\vee} 
    \widehat{f}(\chi)
    \,d\chi 
    =
    6 \kappa 
    \bparen{\prod_{q\equiv2\pmod{3}}
    \sum_{c_q} 
        q^{-2c_qz}}
    \bparen{
        \prod_{p\equiv1\pmod{3}}
        \sum_{a_p,b_p} 
            p^{-(a_p+b_p)z}} . 
\end{equation}
Fix $b_p\geq 0$ and sum over all compatible $a_p$ 
in the right-most sum: 
\begin{equation}\label{eqn:compatibleAp}
    \sum_{a_p\geq -\frac23b_p}
        p^{-(a_p+b_p)z} 
    =p^{-b_pz}
        {\sum_{a_p\geq -\floor{\frac23b_p}}
        p^{-a_pz}}
    ={p^{-(b_p-\lfloor\frac 23b_p\rfloor)z}}
        {\bparen{1-\frac{1}{p^z}}^{-1}}.
\end{equation}
Let $b = 3k+j$ for 
$j \in \{0,1,2\}$ and $k \in \ZZ_{ \geq 0}$. 
Observe that 
\begin{align}
\floor{\tfrac{2}{3}b}=
    \begin{cases}
    2k   &\text{if $b=3k$ or $3k+1$,} \\
    2k+1 &\text{if $b=3k+2$.}
    \end{cases}
\end{align}
Now summing \eqref{eqn:compatibleAp} over $b_p\geq 0$ obtains 
\begin{align}\label{eqn:zetaFactor}
    \sum_{\substack{b_p \geq 0\\a_p\geq -\frac23b_p}} 
        p^{-(a_p+b_p)z} 
        &=\nonumber
        \bparen{1-\frac{1}{p^z}}^{-1}
        \bparen{
            \sum_{b=3k}p^{-kz}
            +\sum_{b=3k+1}p^{-(k+1)z}
            +\sum_{b=3k+2}p^{-(k+1)z}}\\
        &=\nonumber
        \bparen{1-\frac{1}{p^z}}^{-1}
        \left(
        (1-p^{-z})^{-1}
        +2p^{-z}(1-p^{-z})^{-1}
        \right)\\
        &=
        \bparen{1-\frac{1}{p^z}}^{-1}
        \bparen{1+\frac{3}{p^z}\bparen{1-\frac{1}{p^z}}^{-1}}.
\end{align}
Finally we return to finish computing the zeta function. 
Combining 
    \eqref{eqn:zetaFactor} and \eqref{eqn:zetaFunction1} 
    obtains 
\begin{align}
    Z(s)=
    6\kappa 
    \left({1-3^{-z}}\right)
    \zeta(z)
    \prod_{q\equiv2\pmod{3}}
        \bparen{1+\frac{1}{q^{z}}}^{-1} 
    \prod_{p\equiv1\pmod{3}}
        \bparen{1+\frac{3}{p^z}\bparen{1-\frac{1}{p^z}}^{-1}}.
\end{align}
This shows that 
    the restriction of $\widehat{f}$ 
    to $T(\QQ)^\perp\cong (T(\AA)/T(\QQ))^\vee$ is integrable 
    and given by this multivariate Dirichlet series 
    for $\mathrm{Re}(z) = \frac12\mathrm{Re}(s_0+s_1+s_2) \gg 0$. 
The precise region of convergence claimed in the theorem statement 
    will be computed in the lemma below. 

To compute the constant $\kappa$, 
note there is only one monic trace-one cubic polynomial 
of toric height equal to $1$ 
which either has Galois group $C_3$ 
or splits into linear factors over $\QQ$, 
with at most two being the same, 
and it is $f = t^3 - t^2$. 
This polynomial corresponds to 
a unique rational point of $T$ 
since it has repeated factors 
(Proposition~\ref{prop:cyclotomicParametrization}). 
This means the coefficient of $1$ 
in this Dirichlet series is $1$ and 
$\kappa = \frac16$. 
\end{proof} 


%

In the next lemma we reexpress $Z(s)$ 
    in a form better suited for determining the poles and leading constants. 

\begin{lemma}\label{lemma:simplifiedZ(s)} 
    The height zeta function is also given by 
\begin{align}
    Z(s)=\left(1-\frac{1}{3^z}\right)^2
        \zeta_{\QQ(\sqrt{-3})}(z)^2
        \prod_{q\equiv2\pmod{3}}\left(1-\frac{1}{q^{2z}}\right)
        \prod_{p\equiv1\pmod{3}}\left(1-\frac{3}{p^{2z}}+\frac{2}{p^{3z}}\right) 
\end{align}
    where $\zeta_{\QQ(\sqrt{-3})}$ is the Dedekind zeta function 
    of the cyclotomic field $\QQ(\sqrt{-3})$. 
The height zeta function has meromorphic continuation to 
    the region $\{s \in (\CC^{\Sigma(1)})^{\Gamma} : 
    \mathrm{Re}(s_0+s_1+s_2) > 1\}$. 
\end{lemma} 

\begin{proof} 
Since $(1+3x(1-x)^{-1})(1-x)^3 = 1-3x^2+2x^3$ we have 
\begin{align}\label{formula_f(z)}
    \prod_{p\equiv1\pmod{3}}
        \bparen{1+\frac{3}{p^z}\bparen{1-\frac{1}{p^z}}^{-1}}
    \prod_{p\equiv1\pmod{3}}\left(1-\frac{1}{p^z}\right)^{3}
    =
    \prod_{p\equiv1\pmod{3}}
        \left(1-\frac{3}{p^{2z}}+\frac{2}{p^{3z}}\right).
\end{align}
Let $\chi = \left(\frac{-3}{\cdot} \right) 
    = \left(\frac{\cdot}{3} \right)$ 
    be the nontrivial quadratic character of modulus $3$. 
Multiplying both sides of \eqref{formula_f(z)} 
    by $L(z,\chi)$ obtains 
\begin{multline}
\prod_{q\equiv2\pmod{3}}
        \bparen{1+\frac{1}{q^{z}}}^{-1} 
    \prod_{p\equiv1\pmod{3}}
        \bparen{1+\frac{3}{p^z}\bparen{1-\frac{1}{p^z}}^{-1}}
    \prod_{p\equiv1\pmod{3}}
        \left(1-\frac{1}{p^z}\right)^{2}\\
    =
    L(z,\chi)
    \prod_{p\equiv1\pmod{3}}
        \left(1-\frac{3}{p^{2z}}+\frac{2}{p^{3z}}\right).
\end{multline}
Now
\begin{align*}
    &\prod_{p\equiv1\pmod{3}}\left(1-\frac{1}{p^z}\right)^{2}\\[1em]
=&\frac
{\displaystyle
    \prod_{p\equiv1\pmod{3}}\left(1-\frac{1}{p^z}\right)
    \prod_{q\equiv2\pmod{3}}\left(1-\frac{1}{q^z}\right)
    \prod_{p\equiv1\pmod{3}}\left(1-\frac{1}{p^z}\right)
    \prod_{q\equiv2\pmod{3}}\left(1+\frac{1}{q^z}\right)}
{\displaystyle
    \prod_{q\equiv2\pmod{3}}\left(1-\frac{1}{q^{2z}}\right)}\\[1em]
    =&
    \left({\left(1-\frac{1}{3^z}\right)\zeta(z)L(z,\chi)}
        \right)^{-1}
    {\prod_{q\equiv2\pmod{3}}
        \left(1-\frac{1}{q^{2z}}\right)^{-1}}.
\end{align*}
Putting this into the previous equation obtains 
\begin{align} 
    L(z,\chi)
    \prod_{p\equiv1\pmod{3}}
        \left(1-\frac{3}{p^{2z}}+\frac{2}{p^{3z}}\right)
        =
    &\prod_{q\equiv2\pmod{3}}
        \bparen{1+\frac{1}{q^{z}}}^{-1}
    \prod_{p\equiv1\pmod{3}}
        \bparen{1+\frac{3}{p^z}\bparen{1-\frac{1}{p^z}}^{-1}} \\
    &\times \left({\left(1-\frac{1}{3^z}\right)\zeta(z)L(z,\chi)}
        \right)^{-1}
    {\prod_{q\equiv2\pmod{3}}
        \left(1-\frac{1}{q^{2z}}\right)^{-1}}
\end{align} 
which shows that $Z(s)$ is equal to 
\begin{align} 
    &
    (1-3^{-z})\zeta(z) 
    \prod_{q\equiv2\pmod{3}}
        \bparen{1+\frac{1}{q^{z}}}^{-1} 
    \prod_{p\equiv1\pmod{3}}
        \bparen{1+\frac{3}{p^z}\bparen{1-\frac{1}{p^z}}^{-1}}\\
    &= 
    \left(\left(1-\frac{1}{3^z}\right)\zeta(z)L(z,\chi)
    \right)^2
    \prod_{q\equiv2\pmod{3}}
        \left(1-\frac{1}{q^{2z}}\right)
    \prod_{p\equiv1\pmod{3}}
        \left(1-\frac{3}{p^{2z}}+\frac{2}{p^{3z}}\right)\\
    &= 
    \left(\left(1-\frac{1}{3^z}\right)\zeta_{\QQ(\sqrt{-3})}(z)
    \right)^2
    \prod_{q\equiv2\pmod{3}}
        \left(1-\frac{1}{q^{2z}}\right)
    \prod_{p\equiv1\pmod{3}}
        \left(1-\frac{3}{p^{2z}}+\frac{2}{p^{3z}}\right).
\end{align} 

The Dedekind zeta function has meromorphic continuation 
    to the entire complex plane, 
    so the meromorphic continuation of the height zeta function 
    is determined by the remaining Euler product: 
\begin{align}
        \prod_{q\equiv2\pmod{3}}\left(1-\frac{1}{q^{2z}}\right)
        \prod_{p\equiv1\pmod{3}}\left(1-\frac{3}{p^{2z}}+\frac{2}{p^{3z}}\right) 
\end{align}
We have 
\begin{equation} 
    1-x^2=(1+x^2)^{-1}(1-x^4)
\end{equation} 
and
\begin{equation} 
    1-3x^2+2x^3=(1-x^2)^3(1+2x^3-3x^4+O(x^5)).
\end{equation} 
This shows that the Euler product in question is 
\begin{equation} 
    L(2z,\chi)
    \prod_{q\equiv2\pmod{3}}\left(1-\frac{1}{q^{4z}}\right)
    \prod_{p\equiv1\pmod{3}}
        \left(1-\frac{1}{p^{2z}}\right)^4
        \left(1+\frac{2}{p^{3z}}-\frac{3}{p^{4z}}+\cdots\right).
\end{equation} 
The Dirichlet $L$-function is entire. 
The Euler product over ${q\equiv2\pmod{3}}$ 
    is absolutely convergent 
    in the region $\mathrm{Re}(z)>1/4$. 
The Euler product 
    $\prod_{p\equiv1\pmod{3}}
        \left(1-\frac{1}{p^{2z}}\right)^{-4}$ 
    has meromorphic continuation to 
    the region $\mathrm{Re}(z)\geq 1/2$ 
    with a pole of order $2$ when $z = 1/2$ 
    and is nonvanishing on the line $\mathrm{Re}(z) = 1/2$, 
    so $\prod_{p\equiv1\pmod{3}}
        \left(1-\frac{1}{p^{2z}}\right)^{4}$ 
    is holomorphic in the region $\mathrm{Re}(z)> 1/2$. 
The remaining Euler product 
    $\prod_{p\equiv1\pmod{3}}
        \left(1+\frac{2}{p^{3z}}-\frac{3}{p^{4z}}+\cdots\right)$ 
    is absolutely convergent 
    in the region $\mathrm{Re}(z)> 1/3$. 
\end{proof} 

We specialize to the line spanned by $D_0$ 
    in the vector space of toric divisors, 
    and write $$Z_0(s) = Z(sD_0)$$ where $s$ now denotes 
    a single complex variable.  

\begin{proposition}\label{prop:leadingTerms} 
The height zeta function $Z_0(s) = Z(sD_0)$ 
    can be meromorphically continued to the half-plane $\mathrm{Re}(s)>1$ 
    and its only pole in this region is at $s = 2$ with order $2$. 
    Let $$E(s) = 
    \left(1-\frac{1}{3^z}\right)^2
    \prod_{q\equiv2\pmod{3}}
        \left(1-\frac{1}{q^{2z}}\right)
    \prod_{p\equiv1\pmod{3}}
        \left(1-\frac{3}{p^{2z}}+\frac{2}{p^{3z}}\right).$$ 
Then the Laurent expansion of $Z_0(s)$ at $s = 2$ has the form 
\begin{multline} 
    c_2(s-2)^{-2} + c_1(s-2)^{-1} + \cdots \\
    = 4L(1,\chi)^2E(2)(s-2)^{-2} + 
    \bigg(4L(1,\chi)\big(\gamma L(1,\chi) + L'(1,\chi)\big)E(2)
    +4 L(1,\chi)^2 E'(2)\bigg)(s-2)^{-1} + \cdots.
\end{multline} 
Explicitly, 
\begin{equation} 
    c_{2} = 
        \frac{16\pi^2}{243}
        \prod_{q\equiv2\pmod{3}}
            \left(1-\frac{1}{q^{2}}\right)
        \prod_{p\equiv1\pmod{3}}
            \left(1-\frac{3}{p^{2}}+\frac{2}{p^{3}}\right) 
\end{equation} 
and 
\begin{equation} 
    \frac{c_1}{c_2}=
    2 \gamma + \log(2\pi) 
    - 3 \log\left(\frac{\Gamma(1/3)}{\Gamma(2/3)}\right)
    +
    \frac{9}{8}\log 3
    +
    \frac{9}{4}
    \sum_{q\equiv2\pmod{3}}
        \frac{\log q}{q^2-1}
    +
    \frac{27}{4}
    \sum_{p\equiv1\pmod{3}}
        \frac{(p+1)\log p}{p^3-3p+2} .
\end{equation} 
\end{proposition} 

\begin{proof} 
The infinite product for $E(s)$ 
    converges to an analytic function on 
    the half-plane $\mathrm{Re}(s) \geq 2$ 
    so $E(2)$ and $E'(2)$ are well-defined. 
The class number formula gives 
\begin{equation} 
    \lim_{s \to 2} (s-2)\zeta_{\QQ(\sqrt{-3})}(s/2)
    =
    2\cdot
    \frac{2^{r_1}\cdot(2\pi)^{r_2}\cdot R \cdot h}{w \cdot \sqrt{|D|}}
    =
    2\cdot
    \frac{2^{0}\cdot(2\pi)^{1}\cdot 1 \cdot 1}{6 \cdot \sqrt{3}}
    =\frac{2\pi}{3\sqrt{3}}.
\end{equation} 
Thus the coefficient of the leading term is 
\begin{equation} 
    c_2=
    \left(1-\frac{1}{3}\right)^2
        \left(\frac{2\pi}{3\sqrt{3}}\right)^2
        \prod_{q\equiv2\pmod{3}}\left(1-\frac{1}{q^{2}}\right)
        \prod_{p\equiv1\pmod{3}}\left(1-\frac{3}{p^{2}}+\frac{2}{p^{3}}\right) .
\end{equation} 
The coefficient $c_1$ can be computed using 
    the factorization 
    $\zeta_{\QQ(\sqrt{-3})}(z) = \zeta(z) L(z,\chi)$ 
    and \cite[(3.8)]{MR1697650} 
\begin{equation} 
    -L'(1,\chi)=
    \sum_{n = 2}^\infty 
        \frac{\chi(n) \log n}{n}
    =
    \frac{\pi}{\sqrt{3}}
    \left(
    \log\left(\frac{\Gamma(1/3)}{\Gamma(2/3)}\right)
    -
    \frac13 (\gamma + \log(2\pi)) 
    \right).\qedhere
\end{equation} 
\end{proof} 

\begin{proof}[Proof of Theorem~\ref{thm:heightZetaFunctionFormula}] 
The expression for $Z_0(s)$ follows 
    from combining 
    Lemma~\ref{lemma:simplifiedZ(s)} and 
    Proposition~\ref{prop:leadingTerms}. 
Fix the isomorphism $\mathrm{Pic}(S) \otimes \QQ \to \QQ$ 
    taking the ample generator to $3$. 
It remains to be seen that 
    the image of the line spanned by $D_0$ in 
    the vector space of toric divisors 
    is identified with $\mathrm{Pic}(S) \otimes \CC$ 
    such that $D_0$ corresponds to $s = 1$. 
The canonical divisor $K$ is $D_0+D_1+D_2$. 
The surface $S$ has Picard rank one 
    \cite[Corollary~3.6]{odesky2023moduli} 
    and the unique ample generator 
    is equivalent up to torsion in 
    the divisor class group to $-K$ 
    by \cite[Theorem~3.5]{odesky2023moduli} 
    and \cite[Theorem~3.7]{odesky2023moduli}. 
One computes that $3D_0$ is linearly equivalent to $K$ 
    so $\mathcal O(D_0)=\frac13\mathcal O(K)\leftrightarrow s =1$. 
\end{proof} 


\section{Proofs of Theorem~\ref{thm:asymptoticPolynomialCount} and Theorem~\ref{thm:exactFormula}} 

\begin{lemma}\label{lemma:cubicimpliesnormal}
Let $x \in \mathbb C$ be a root 
of an irreducible polynomial 
with rational coefficients 
with Galois group $C_3$ 
    and $t^2$-coefficient $-1$. 
Then $x$ is a normal element in 
the Galois extension $\mathbb Q(x)/\mathbb Q$. 
\end{lemma}

\begin{proof} 
Let $\sigma$ be a generator for $C_3$ and 
set $y = \sigma x$, $z = \sigma^2 x$. 
Suppose for the sake of contradiction that 
    the points $x,y$ and $z$  
    lie on a plane $P$ in $\QQ(x) \otimes \RR$ containing $0$. 
Then $x,y$ and $z$ lie on a line $L$, 
    namely the intersection of $P$ with 
    the trace-one affine hyperplane 
    $\{\mathrm{tr}^{\QQ(x)}_\QQ = 1\}$. 
This implies that $z-y=\sigma(y-x)$ 
    is proportional to $y-x$, 
    and thus $y-x$ is an eigenvector of $\sigma$. 
The only real eigenvalue of $\sigma$ is one, 
    so $y-x=z-y=x-z$, 
    all equal to some nonzero element $\lambda$ of $\QQ$. 
Adding these up shows that 
    $y+z+x-x-y-z = 0 = 3 \lambda$, a contradiction. 
\end{proof} 

\begin{lemma}\label{lemma:integralTraceOneImpliesNormal}
Let $(\alpha,\beta,\gamma) \in \ZZ^3$ 
    satisfy $\alpha+\beta+\gamma=1$. 
Then $(\alpha,\beta,\gamma)$ is a normal element in 
the split $\mathbb Q$-algebra $\QQ^3$. 
\end{lemma}

\begin{proof} 
If $(\alpha,\beta,\gamma)$ is not normal, then
\begin{equation}\label{eqn:1}
\det\begin{bmatrix}
        \alpha&\beta&\gamma\\\beta&\gamma&\alpha\\\gamma&\alpha&\beta
    \end{bmatrix}=3\alpha\beta\gamma-\alpha^3-\beta^3-\gamma^3=0. 
\end{equation}
Set $a=\alpha\beta+\beta\gamma+\gamma\alpha$. First observe that 
    $1=(\alpha+\beta+\gamma)^2=\alpha^2+\beta^2+\gamma^2+2a$ 
    and so $\alpha^2+\beta^2+\gamma^2=1-2a$. 
Next, 
\begin{align*}
    a=(\alpha\beta+\beta\gamma+\gamma\alpha)(\alpha+\beta+\gamma)&=3\alpha\beta\gamma+\alpha^2(\beta+\gamma)+\beta^2(\alpha+\gamma)+\gamma^2(\alpha+\beta)\\
    &=3\alpha\beta\gamma+\alpha^2(1-\alpha)+\beta^2(1-\beta)+\gamma^2(1-\gamma)\\
    &=3\alpha\beta\gamma+\alpha^2+\beta^2+\gamma^2-(\alpha^3+\beta^3+\gamma^3).
\end{align*}
Putting these together with \eqref{eqn:1} shows that 
\begin{equation} 
    a=3\alpha\beta\gamma-\alpha^3-\beta^3-\gamma^3+(1-2a) = 1-2a
\end{equation} 
which is impossible since $a\in\ZZ$. 
\end{proof} 

A polynomial $f = t^3 -t^2 + at + b  
= (t-\alpha)(t-\beta)(t-\gamma)\in \ZZ[t]$ 
which splits into three linear factors over $\QQ$ 
will be called \emph{normal} if 
$x = (\alpha,\beta,\gamma)$ is a normal element of the split $C_3$-algebra 
$K_{\mathrm{spl}} = \QQ^3$. 
Since $x$ is normal if and only if it has at most two identical coordinates, 
the split polynomial $f$ is normal if and only if 
it has at most two identical roots. 
From the above lemmas we see that the polynomials 
under consideration are all normal, 
which means they are all realized by rational points of $T$. 

\begin{corollary}\label{cor:automaticallyNormal}
Let $F$ denote the set of polynomials 
    of the form $t^3 -t^2 + at + b \in \ZZ[t]$ 
    which either have Galois group $C_3$ or 
    split into three linear factors over $\QQ$. 
Then any $f \in F$ is normal. 
\end{corollary}


\begin{lemma}\label{lemma:reducibles}
    We have 
\begin{equation} 
    \#\{f \in F : \text{reducible, $\mathrm{disc}(f) \neq 0$}, H(f) \leq H\}
    =
    \frac{\pi}{9\sqrt{3}}H^2-\tfrac16 H+O(H^{t})
\end{equation} 
for some $\frac 12<t<1$ 
and 
\begin{equation} 
    \#\{f \in F : \text{reducible, $\mathrm{disc}(f) = 0$}, H(f) \leq H\}
    =
    \tfrac13H + O(1). 
\end{equation} 
\end{lemma}

In the error term 
    one may take $t = \frac{131}{208}$ \cite{MR1956254}. 

\begin{figure}[h!]
    \centering
    \includegraphics[scale=0.3]{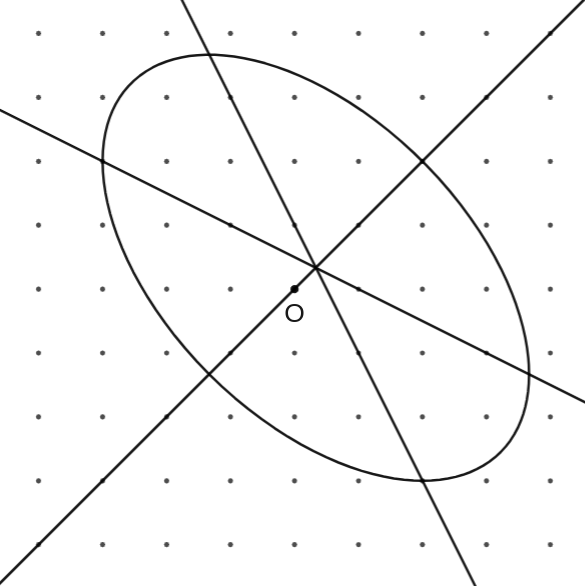}
    \caption{The ellipse $E_5$ and the three lines of points with nontrivial stabilizer in $S_3$.}
\end{figure} 

\begin{proof} 
Consider the ellipse in $\RR^2$ defined by 
\begin{equation} 
    E_H : 
    -a=x^2+y^2+xy-x-y = \tfrac{1}{3}(H^2-1).
\end{equation} 
The permutation action of the symmetric group $S_3$ 
stabilizes the affine hyperplane $x+y+z = 1$. 
If we identify $\RR^2$ with this affine hyperplane 
    via $(x,y) \mapsto (x,y,1-x-y)$ then 
    the induced action of $S_3$ on $\RR^2$ 
    stabilizes the level sets of $x^2+y^2+xy-x-y$,  
    and therefore acts on the interior of $E_H$. 
Let $E_H^\circ = E_H \cup \mathrm{int}(E_H)$. 
Then we have a canonical bijection 
\begin{equation} 
    (E_H^\circ \cap \ZZ^2)/S_3 
    \xrightarrow{\sim} \{f \in F : \text{reducible},\,\,\, H(f) \leq H\}.
\end{equation} 
 
A lattice point $(\alpha,\beta)$ has a nontrivial stabilizer in $S_3$ 
    if and only if either 
    $\alpha = \beta$ or 
    $1 - \alpha - \beta \in \{\alpha,\beta\}$, 
so the number of lattice points in $E_H^\circ$ 
    with a nontrivial stabilizer is $H + O(1)$. 
The area of $E_H^\circ$ is 
    $A_H=\frac{2\pi}{3\sqrt{3}}(H^2-1)$, 
so the number of lattice points in $E_H^\circ$ 
    is $A_H+O(H^{t})$ for some $t < 1$ 
    (conjecturally $t = \tfrac12+\varepsilon$). 
    Thus 
\begin{equation} 
    \#\{f \in F : \text{reducible, $\mathrm{disc}(f) \neq 0$}, H(f) \leq H\}
    =
    \tfrac16 (A_H-H+O(H^{t}))
\end{equation} 
and 
\begin{equation} 
    \#\{f \in F : \text{reducible, $\mathrm{disc}(f) = 0$}, H(f) \leq H\}
    =
    \tfrac13H + O(1). \qedhere
\end{equation} 
\end{proof} 

Theorem~\ref{thm:exactFormula} is now easily proven 
    by subtracting off the count for reducible polynomials in Lemma~\ref{lemma:reducibles} 
    from the Dirichlet coefficients of $Z(s)$ 
    as expressed in Theorem~\ref{thm:heightZetaFunctionFormula2}. 
We may now prove Theorem~\ref{thm:asymptoticPolynomialCount}. 

\begin{proof}[Proof of Theorem~\ref{thm:asymptoticPolynomialCount}] 
Let $d_n$ denote the $n$th Dirichlet coefficient of $Z_0(2s)$. 
By the modular interpretation for $\mathcal G/C_3$ 
    (Theorem~\ref{thm:modularInterpretation}), 
    $d_n$ is equal to the number of equivalence classes $(K,x)$ 
    of Galois $C_3$-algebras $K/\mathbb Q$ 
    equipped with a trace one normal element $x \in K$ 
    and toric height $\sqrt{n}$. 
Let $K'$ denote the twist of the $C_3$-algebra $K$ 
    by the outer automorphism of $C_3$. 
Each rational point $(K,x)$ falls into one of the following cases 
    (cf.~examples from \S\ref{sec:orbitParam}): 
\begin{enumerate} 
\item $K$ is an abelian cubic field, 
\item $K$ is the split $C_3$-algebra $K_{\text{spl}}=\QQ^3$ 
    and $x$ has exactly two identical coordinates, or 
\item $K$ is the split $C_3$-algebra $K_{\text{spl}}=\QQ^3$ 
    and $x$ has distinct coordinates. 
\end{enumerate} 
(It cannot happen that $K = K_{\text{spl}}$ 
    and $x$ has three identical coordinates 
    since $x$ would not be normal.) 
In these cases, respectively, we have 
\begin{enumerate} 
\item $K \not \cong K'$ and $(K,x) \neq (K',x)$, 
\item $K \cong K'$ and $(K,x) = (K',x)$, or 
\item $K \cong K'$ and $(K,x) \neq (K',x)$. 
\end{enumerate} 
The characteristic polynomial $f$ of $x$ 
    nearly determines the rational point $(K/\QQ,x)$ --- 
    in these cases, respectively, $f$ arises as 
    the characteristic polynomial for 
\begin{enumerate} 
\item precisely the two rational points $(K,x)$ and $(K',x)$, 
\item only the rational point $(K_{\text{spl}},x)$, or 
\item precisely the two rational points 
    $(K_{\text{spl}},x)$ and 
        $(K_{\text{spl}}',x)$.\footnote{Let $\sigma$ be a transposition in $S_3$. 
        Then $(K_{\text{spl}},\sigma x)$ has the same characteristic polynomial 
        as $(K_{\text{spl}},x)$ but it does not give us another rational point 
        since 
        $(K_{\text{spl}}',x)=(K_{\text{spl}},\sigma x)$. 
        Thus these two rational points 
        account for $(K_{\text{spl}},\sigma x)$ for any $\sigma \in S_3$.} 
\end{enumerate} 

Let $F$ denote the set of polynomials 
    $t^3 -t^2 + at + b \in \ZZ[t]$ 
    which either have Galois group $C_3$ or 
    split into three linear factors over $\QQ$. 
Then any $f \in F$ is automatically normal 
    (Corollary~\ref{cor:automaticallyNormal})  
    so arises as the characteristic polynomial 
    for some rational point in $T$. 
The preceding analysis shows that 
    the number $w_f$ of rational points of $T$ 
    with characteristic polynomial equal to 
    a given $f \in F$ 
    is given by \eqref{eqn:polynomialWeights}. 
Thus among $f \in F$ with $H(f) = \sqrt{n}$ 
we have that 
\begin{align*} 
    2\#\{\text{irreducible}\}=
    d_n-
    \#\{\text{reducible, }\mathrm{disc}(f) = 0\}
    -2\#\{\text{reducible, }\mathrm{disc}(f) \neq 0\}.
\end{align*} 
Now we sum over $f$ with $H(f) \leq H$. 
Then 
\begin{equation} 
    2\sum_{F_{\mathrm{irr}}, H(f) \leq H} 1
    = 
    \sum_{n \leq H^2}d_n
    -\sum_{\substack{F_{\mathrm{red}}, H(f) \leq H\\\mathrm{disc}(f) = 0}} 1
    -2\sum_{\substack{F_{\mathrm{red}}, H(f) \leq H\\\mathrm{disc}(f) \neq 0}} 1.
\end{equation} 
By Lemma~\ref{lemma:reducibles} this is 
\begin{equation} 
    \sum_{n \leq H^2}d_n
    -\tfrac13H 
    -2\left(\frac{\pi}{9\sqrt{3}}H^2-\tfrac16 H
    +O(H^{t})\right)
    =
    \sum_{n \leq H^2}d_n
    -\frac{2\pi}{9\sqrt{3}}H^2
    +O(H^{t}).
\end{equation} 
Applying standard Tauberian theorems to $Z_0(s)$ 
and using the information about the poles 
    and meromorphic continuation 
    in Lemma~\ref{lemma:simplifiedZ(s)} and Proposition~\ref{prop:leadingTerms} 
    shows that 
\begin{equation} 
    \sum_{n \leq H^2}d_n
    =
    \tfrac 12c_2H^2 \log H + \tfrac12 c_1 H^2 + O_\varepsilon(H^{1+\varepsilon})
\end{equation} 
for any $\varepsilon > 0$. 
Putting this all together shows that 
\begin{equation}\label{eqn:irrPolys} 
    \sum_{F_{\mathrm{irr}}, H(f) \leq H} 1
    = 
    \tfrac 14c_2H^2 \log H + \tfrac14c_1 H^2 
    -\frac{\pi}{9\sqrt{3}}H^2
    + O_\varepsilon(H^{1+\varepsilon}).\qedhere
\end{equation} 
This is the asymptotic count for polynomials of bounded toric height. 
By the comparison between toric height and root height 
    (Remark~\ref{rmk:rootVsToric}), 
    the asymptotic count for polynomials of bounded root height 
    is obtained by replacing $H$ with $\sqrt{3}H$. 
\end{proof} 

\begin{rem}
By the Riemann hypothesis 
    one expects 
    $\prod_{p\equiv1\pmod{3}}
        \left(1-\frac{1}{p^{2z}}\right)^{4}$ 
    to have analytic continuation 
    to the region $\mathrm{Re}(z) > \frac14$, 
    and also for $\zeta_{\QQ(\sqrt{-3})}(z)$ 
    to be nonvanishing at $z = 1/3$, 
    in which case the $O_\varepsilon(H^{1+\varepsilon})$ 
    in \eqref{eqn:irrPolys} 
    should in fact be 
    $aH^{2/3}\log H + b H^{2/3}+O(H^{t})$ 
    for some computable nonzero constants $a,b$ 
    where $t = \frac{131}{208}$ \cite{MR1956254} 
    is the best known exponent for the error term in 
    the Gauss circle problem. 
\end{rem}


\bibliography{draft}
\bibliographystyle{abbrv}

\end{document}